\documentclass[12pt]{amsart}
\usepackage{amsfonts,amssymb,eucal}

\newcommand{\GM}{\ensuremath\tilde{{\Gamma}}}
\usepackage{color}
\usepackage{amsthm} 
 \usepackage{amsmath} 
 \usepackage{latexsym}
\usepackage{bbold,mathbbol,bbm}
\numberwithin{equation}{section}
\newcommand{\ep}{\ensuremath{\epsilon\,}}
\newcommand{\te}{\ensuremath{\tau}}

\newcommand{\mb}[1]{{\mbox{\boldmath{$#1$}}}}
\newcommand{\mc}[1]{{\mathcal{#1}}}
\newcommand{\got}[1]{{\mathfrak{#1}}}
\newcommand{\db}[1]{{\mathbb{#1}}}
\newcommand{\pa}{\partial}

\newcommand{\R}{\ensuremath{\mathbb{R}}}
\newcommand{\C}{\ensuremath{\mathbb{C}}}
\newcommand{\N}{\ensuremath{\mathbb{N}}}

\newcommand{\Hi}{\ensuremath{\got{H}}}

\newcommand{\g}{\ensuremath{\got{g}}}
\newtheorem{Remark}{Remark}

\newtheorem{Theorem}{Theorem}
\newtheorem{Proposition}{Proposition}

\theoremstyle{definition} 

 \newcommand{\SL}{\ensuremath{{\mbox{\rm{SL}}(2,\R)}}}
  \newcommand{\Spp}{\ensuremath{{\mbox{\rm{Sp}}(2,\R)}}}
\newcommand{\Ka}{K\"ahler}
\newcommand{\mr}[1]{{\mathrm{#1}}}
\def\ii{\operatorname{i}}
\newcommand{\dd}{\operatorname{d}}
\newcommand{\un}{\ensuremath{\mathbb{1}_n}}

\newtheorem{deff}{Definition}
\renewcommand{\Re}{\operatorname{Re}}
\renewcommand{\Im}{\operatorname{Im}}
\allowdisplaybreaks 
\newcommand{\Sp}{\ensuremath{{\mbox{\rm{Sp}}(n,\R)}}}

\begin{document}
\title{Geodesics on the  extended  Siegel--Jacobi upper half-plane}
\author{Stefan  Berceanu}
\address[Stefan  Berceanu]{National
 Institute for Physics and Nuclear Engineering\\
         Department of Theoretical Physics\\
         PO BOX MG-6, Bucharest-Magurele, Romania}
       \email{Berceanu@theory.nipne.ro}
     \begin{abstract}
  The semidirect product of the real Heisenberg group ${\rm H}_1(\mathbb{R})$ with ${\rm
    SL}(2,\mathbb{R})$, called the  real Jacobi group
  $G^J_1(\mathbb{R})$, admits a four-parameter
  invariant metric expressed in the S-coordinates. We  determine
  the geodesic equations on the extended  Siegel--Jacobi upper
  half-plane $\tilde{\mathcal{X}}^J_1=\frac{G^J_1(\R)}{\rm{SO}(2)}\approx
   \mathcal{X}^J_1\times \mathbb{R}\approx \mathcal{X}_1\times
   \mathbb{R}^3$, where  $\mathcal{X}^J_1$ ($\mathcal{X}_1)$ denotes the Siegel-Jacobi
   upper half-plane (respectively Siegel upper half-plane). Equating successively  with zero the values of
   the three parameters in the geodesic equations on
   $\tilde{\mathcal{X}}^J_1$, we get the geodesic equations on 
    $\mathcal{X}^J_1$, $\mathcal{X}_1$ and ${\rm H}_1(\mathbb{R})$.
\end{abstract}

\subjclass{53C22,32F45,53C55,53C30,81R30}
\keywords{Jacobi group, invariant metric,  Siegel--Jacobi disk, Siegel--Jacobi upper
  half-plane,  extended
  Siegel--Jacobi upper half-plane, geodesics,
 coherent states, geodesic mapping}
\maketitle

\tableofcontents
\section{Introduction}

The real Jacobi group of degree $n$ is defined  as  $G^J_n(\R):={\rm
  Sp}(n,\R)\ltimes \mr{H}_n(\R)$ \cite{bs,ez,tak,ZIEG}, where  $\rm{H}_n(\R)$ denotes the
$(2n+1)$-dimensional real Heisenberg group,  while  the semidirect product
$\mr{H}_n\rtimes{\rm Sp}(n,\R)_{\C}$, ${\rm Sp}(n,\R)_{\C}:= {\rm Sp}(n,\C)\cap {\rm U}(n,n)$, is 
 denoted by $G^J_n$ \cite{sbj,nou}.
Both Jacobi groups
$G^J_n(\R)$ and $G^J_n$ are  intensively  investigated in Mathematics, 
Mathematical Physics and Theoretical Physics \cite{jac1,SB15,
  SB19,BERC08B,gem,bs,ez},
\cite{yang}-\cite{Y10}.

The Siegel-Jacobi  upper half space is the  $G^J_n(\R)$-homogeneous
manifold 
$\mc{X}^J_n:=\frac{G^J_n(\R)}{\rm{U}(n)\times \R}$
 $\approx \mc{X}_n\times
\R^{2n}$ \cite{SB15,SB19,yang,Y08}, where  $\mc{X}_n$ denotes the Siegel upper half space realized as 
 the noncompact Hermitian symmetric space $ \frac{\Sp}{\rm{U}(n)}$ \cite[p 398]{helg}.
The Siegel-Jacobi  ball  is $\mc{D}^J_n:=  \frac{G^J_n}{\rm{U}(n)\times \R}\approx \C^n\times\mc{D}_n$ \cite{sbj},  where  $\mc{D}_n\approx  \operatorname{Sp}(n, \R
)_{\C}/\operatorname{U}(n)$ denotes the Siegel (open)  ball  of degree
$n$ \cite{helg}.

The Jacobi group  $G^J_n$ is a unimodular, non-reductive, algebraic group of
Harish-Chandra type  \cite{gem,LEE03,SA71,SA80}, and 
 $\mc{D}^J_n$  is a  reductive,
non-symmetric manifold \cite{SB15,SB19}. 
The holomorphic irreducible unitary representations of $G^J_n$ 
based on $\mc{D}^J_n$    constructed in 
 \cite{gem,bs,tak,TA99} 
are   relevant   in  several  areas of mathematics such as Jacobi forms, automorphic forms,
L-functions and modular forms, spherical functions, the ring of
invariant differential operators, 
theta functions, Hecke operators, Shimura varieties and  Kuga fiber
varieties.

 The groups $G^J_n(\R)$, ${\rm Sp}(n,\R)$,
$\mr{H}_n(\R)$, and the K\"ahlerian  homogenous manifolds
$\mc{X}^J_n$
 and $\mc{X}_n$ are isomorphic with $G^J_n$, ${\rm Sp}(n,\R)_{\C}$, $\mr{H}_n$,
$\mc{D}^J_n$, 
respectively  $\mc{D}_n$, see  \cite{sbj,FC,gem,
  bs,ez,yang,Y07}. We denote with the same symbol $\mr{H}_n$ both
isomorphic groups $\mr{H}_n$ and $\mr{H}_n(\R)$.

The Jacobi group was investigated 
 \cite{lis2,neeb96,neeb} via 
 coherent states (CS) \cite{mosc,mv,perG}. 
 CS-systems based on the Siegel-Jacobi ball  have applications
in  quantum mechanics, geometric quantization,
dequantization, quantum optics, squeezed states, quantum
teleportation,  quantum tomography,  nuclear structure,  signal  processing, 
Vlasov kinetic equation, see references in   \cite{SB20,GAB,SB19}. 

Applying  Berezin's  procedure  \cite{ber73}-\cite{berezin}  to obtain  balanced 
 metric on homogenous  K\"ahler
manifolds \cite{arr, don,alo}, we have determined the two-parameter
invariant metric on the $G^J_n$ ($G^J_n(\R)$)-homogenous  manifold
$\mc{D}^J_n$ (respectively $\mc{X}^J_n$), firstly for $n=1$ in \cite{jac1,BER77} and
then for any  $n\in \N$ in 
\cite{sbj,nou,SB14}, see also  the papers of Yang \cite{Y07,Y08,Y10}
and \cite[Section 5.1]{nou}.  In order to determine three-parameter invariant metric on the
five-dimensional manifold $\tilde{\mc{X}}^J_1$, called
extended Siegel-Jacobi half-plane  \cite{SB19}, we have
abandoned the CS-method in the favor of Cartan's moving frame
method \cite{cart4,cart5,ev},  considering $G^J_1(\R)$ embedded in
$\Spp$ \cite{bs,ez}. In  \cite{SB20N} we have
 determined the three-parameter invariant metric on the extended Siegel-Jacobi upper half
space $\tilde{\mc{X}}^J_n$, $n\in \N$, a generalization of the
metric on $\tilde{\mc{X}}^J_1$ obtained in \cite{SB19}.

 In our studies on the geometry of the Siegel-Jacobi disk by the
 CS-method,  in \cite {jac1, FC, SB14} we paid special attention to
 the determination of the geodesic equations on $ \mc {D} ^ J_1 $,
 while in   \cite{GAB}
 we calculated  
geodesics on the Siegel-Jacobi ball $\mc{D}^J_n$. Equating successively
with zero  parameters in 
 the four-parameter    invariant metrics on  $G^J_1(\R)$ expressed in
 the S-coordinates $(x,y,\theta,p,q,\kappa)$ \cite{bs,ez}, we obtained
 in \cite{SB19}
 the invariant metric on $\tilde{\mc{X}}^J_1$, $\mc{X}^J_1$,
 $\mc{X}_1$.

 The present paper, 
devoted to  geodesic equations on homogeneous manifolds
associated with  the real Jacobi group $G^J_1(\R)$,   is a
continuation of \cite{SB20,SB19}. Using the results obtained in \cite{SB19} and applied in
 \cite{SB20},  we determine
the  geodesic equations on homogeneous manifolds
associated to   the real Jacobi group $G^J_1(\R)$ in the spirit of
\cite{jac1,FC,SB14, GAB}.  Although we are not able to integrate the
 geodesic equations,  we make important
remarks on geometric and physical meaning of the variables used.  Taking successively the values  
of the four parameters of the invariant metric on $G^J_1(\R)$ equal with zero in the  geodesic equations 
on the extended Siegel-Jacobi upper half plane, we get the  geodesic equations
 on the Siegel-Jacobi upper
half-plane,  Siegel upper half-plane and Heisenberg group.

The paper is laid out as follows. Section \ref{PRE} recalls known
formulas on  coherent states, balanced metric,  geodesics, and describes the
parametrization of $G^J_1(\R)$ with  the S-coordinates.
Proposition 
\ref{BIGTH} recalls the four-parameter invariant metric of
$G^J_1(\R)$ obtained in \cite[Theorem 5.7]{SB19}. 
Section \ref{GSJD} summarizes our 
previous results on the invariant \Ka~ two-form  $\omega_{\mc{D}^J_1}
(w,z)$, where $(w,z)\in (\mc{D}_1,\C)$. Proposition
\ref{PRFC}, taken  from \cite{SB20,SB19}, expresses the effect of the
change of coordinates  $FC: (w,z)\rightarrow
(v,\eta)$, partial Cayley transform $\Phi:(w,z)\rightarrow
(v,u)\in\mc{X}^J_1$, $FC_1:(v,u)\rightarrow (v,\eta)$ and $(v,u)\rightarrow (x,y,p,q)$ applied to 
 the invariant \Ka~ two-form on
$\mc{D}^J_1$.  Proposition
\ref{STR}, a completion of Proposition \ref{PRFC}, gives the change
of coordinates  $\Phi_1:=FC_1\circ\Phi:~\mc{D}^J_1\ni (w,z)\rightarrow(v,\eta)=(x+\ii
y,q+\ii p)\in \mc{X}^J_1 $,  called  the second  partial
Cayley transform.  
The   geodesic equations  on $\mc{D}^J_1$ in the complex variables $(w,z)$ are given  in Proposition
\ref{GER} and Remark \ref{RE2}, while Proposition \ref{PP55} gives the
geodesic equations  in the real and imaginary parts of the complex
variables $w,z$. In Proposition \ref{P3} and Remark
\ref{RE3} in   Section \ref{GSJUH} the information on geodesics on
$\mc{D}^J_1$ in Section \ref{GSJD} is carried   in the information on geodesics on
$\mc{X}^J_1$ via the partial  Cayley transform,  which is proved to be
a geodesic mapping \cite[Definition 5.1 p 127]{mikes}. Geodesics on
$\mc{X}^J_1$ expressed in the variables $(x,y,p,q)$ are given in
Proposition \ref{P4}.  Proposition \ref{PP5} proves by brute force
calculation that the second partial Cayley transform $\Phi_1$
is a geodesic mapping.  Section \ref{GEX} is devoted to the geodesics
on the extended Siegel-Jacobi upper half-plane  
$\tilde{\mc{X}}^J_1$ corresponding to the three-parameter invariant
metric in the S-variables $(x,y,p,q,\kappa)$ determined in
\cite{SB19}. To help the reader,  in  Section \ref{gm} are
recalled   some
facts  on geodesic mappings extracted from \cite{alx,EISEN,helg,mikes}.

The new results obtained in this paper are
contained in the  Remarks \ref{RM10}, \ref{RE3}, 
  item 6 of Proposition  \ref{BIGTH}, Propositions 
  \ref{STR}--\ref{PP5}
and Theorem \ref{PROP5}, which is the main result of the present
work.

{\bf Notation}

We denote by $\mathbb{R}$, $\mathbb{C}$, $\mathbb{Z}$ and $\mathbb{N}$ 
 the field of real numbers, the field of complex numbers,
the ring of integers,   and the set of non-negative integers, respectively. We denote the imaginary unit
$\sqrt{-1}$ by~$\ii$, the real and imaginary parts of a complex
number $z\in\C$ by $\Re z$ and $\Im z$ respectively, and the complex 
conjugate of $z\in\C$ by $\bar{z}$. We denote by $\det(M)$ the
determinant of the  matrix~$M$. $M(n,m,\db{F})$ denotes the set
of $n\times m$ matrices with entries in the field $\db{F}$. We denote
by $M(n,\db{F})$ the set $M(n,n,\db{F})$.
$\un$ denotes the unit matrix in $M(n,\db{F})$.
 We denote by ${\dd }$ the differential. 
We use Einstein convention i.e.  repeated indices are
implicitly summed over. The scalar product of vectors in the Hilbert
space  $\got{H}$ is
denoted $(\cdot,\cdot)$.  If $f$ is a real or  complex function of
$t\in\R$, then sometimes we use the  standard abbreviations  $\dot{f}:=\frac{\dd f}{\dd t}$,
$\ddot{f}:=\frac{\dd^2 f}{\dd t^2}$. We also denote $\pa_i:=\frac{\pa}{\pa
  x^i}, i=1,\dots,n.$

\section{Preliminaries}\label{PRE}

 \subsection{Balanced metric and coherent states}\label{BMES}
In Perelomov's approach
to CS  \cite{perG} it is supposed that there exists 
 a continuous, unitary, irreducible 
representation $\pi$
 of a   Lie group $G$
 on a   separable  complex  Hilbert space \Hi. 
 If $H$ is the isotropy group of the representation $\pi$, then 
two types of CS-vectors belonging to  $\got{H}$ are locally defined on
$M=G/H$:  the normalized (un-normalized) CS-vector
$\underline{e}_x$ (respectively, $e_z$) \cite{SB03,perG}
 \begin{equation}\label{2.1}
\underline{e}_x=\exp(\sum_{\phi\in\Delta^+}x_{\phi}{\mb{X}}^+_{\phi}-{\bar{x}}_{\phi}{\mb{X}}^-_{\phi})e_0,
\quad e_z=\exp(\sum_{\phi\in\Delta^+}z_{\phi}{\mb{X}}^+_{\phi})e_0,
\end{equation}
where $e_0$ is the extremal weight vector of the representation $\pi$,
$\Delta^+$ is the set of positive roots
of the Lie algebra $\got{g}$, and 
$X^+_{\phi}$   ($X^-_{\phi}$) 
 are the positive (respectively, negative) generators. For  $X\in
 \got{g}$  we denoted in \eqref{2.1} $\mb{X}:=\dd\pi(X)$  \cite{SB03,SB14,perG}.

We  denoted  by $FC$ \cite{nou} the change of variables $x \rightarrow  z$ in
formula \eqref{2.1} such that
\begin{equation}\label{2.2}
\underline{e}_x =  (e_z,e_z)^{-\frac{1}{2}}e_z; \quad z = FC(x).
\end{equation}

At first our interest in geodesics  on homogenous spaces $ M=G/H$ in the
context of CS came from the fact that for symmetric manifolds $M$ the
FC-transform   explicitly provides  the solution of geodesic equations on $M$ through the identity of $M$
\cite{ber97A,sbl}. Next
we showed that this property is still true for naturally reductive
manifolds \cite{ber97}. But the non-symmetric manifold  $\mc{X}^J_1$ $(\tilde{\mc{X}}^J_1)$ is  not naturally
reductive  with respect to the balanced metric \cite{SB19}
(respectively,  the three-parmeter invariant metric \cite{SB20}). 

We consider  a  $2n$-dimensional K\"ahler manifold  $M=G/H$
endowed with a
$G$-invariant 
K\"ahler two-form 
\begin{equation}\label{kall}
\omega_M(z)=\ii\sum_{\alpha,\beta=1}^n h_{\alpha\bar{\beta}} (z) \dd z_{\alpha}\wedge
\dd\bar{z}_{\beta}, ~h_{\alpha\bar{\beta}}= \bar{h}_{\beta\bar{\alpha}}= h_{\bar{\beta}\alpha},
\end{equation}
 derived from the  K\"ahler potential $f(z,\bar{z})$ \cite{chern}
\[h_{\alpha\bar{\beta}}= \frac{\pa^2 f}{\pa {z}_{\alpha}\pa
  \bar{z}_{{\beta}}} .
\]
 The condition for  the Hermitian metric to be a K\"ahlerian one is, cf.  \cite[(6) p 156]{kn}
\begin{equation}\label{condH}
\frac{\pa h_{\alpha\bar{\beta}}}{\pa z_{\gamma}} = \frac{\pa
  h_{\gamma\bar{\beta}}}{\pa z_{\alpha}},~\alpha,\beta,\gamma
=1,\dots, n .
\end{equation}
As was pointed out  in \cite{berr} for $\mc{D}^J_1$ and proved in
\cite{SB15} for $\mc{D}^J_n$, $n\in \N$, the homogeneous  hermitian  metrics determined in \cite{jac1,sbj,nou}
are  actually  balanced metrics \cite{arr,don}, because they  come from  the K\"ahler
potential calculated as the scalar  product of two CS-vectors 
\[
f(z,\bar{z})=\ln K_M(z,\bar{z}), ~~K_M(z,\bar{z}):=(e_{\bar{z}},e_{\bar{z}}).
\]

\subsection{Geodesics on Riemannian and  \Ka~
 manifolds}\label{GD}

Let $(M,g)$ be a $n$-dimen-\\sional  Riemannian manifold. In  a local coordinate system $x^1,\dots$, $x^n$ the geodesic equations on a manifold $M$ 
with components of the linear connection $\Gamma$ are, see
e.g.  \cite[Proposition 7.8 p 145]{kn1}
\begin{equation} \label{GEOO}
\frac{\dd ^2 x^i}{\dd t^2}  +\sum_{j,k=1}^n\Gamma^i_{jk}\frac{\dd x^j}{\dd t}
\frac{\dd x^k}{\dd t}  =  0, ~i=1,\dots,n .
\end{equation}
The components $\Gamma^i_{jk}$ (Christofell's symbols of second kind) of a Riemannian (Levi-Civita)
connection $\nabla$  
are obtained from the  Christofell's symbols of first kind $[ij,k]$ by solving  the linear system of algebraic equations, see e.g. \cite[p 160]{kn1}
\begin{equation}\label{geot}g_{lk}\Gamma^l_{ji}=[ij,k]=:\frac{1}{2}\left(\frac{\pa
      g_{ki}}{\pa x_j}+\frac{\pa g_{jk}}{\pa x_i}-\frac{\pa g
      _{ji}}{\pa x_k}\right),~~i,j,k,l=1,\dots,n. \end{equation}
The $\frac{n^2(n+1)}{2}$ distinct $\Gamma$-symbols of an
 $n$-dimensional  manifold are given by the formulas
\begin{equation}\label{geoI}
  \Gamma^m_{ij}=[ij,k]g^{km}, \quad
\text{where~ ~ ~ }g^{mk}g_{kl}=\delta^m_l, ~~i,j,k,m=1,\dots,n.
\end{equation}

In the convention  $\alpha, \beta, \gamma,\dots$
run from 1 to $n$, while A,B,C,$\dots$  run through $1,\dots,n, $  
 $\bar{1},\dots,\bar{n}$, \cite[p 155]{kn}, for an almost complex
 connection without torsion we have
the relations 
$$\Gamma^{\alpha}_{\beta\gamma}=\Gamma^{\alpha}_{\gamma \beta}; \quad
\bar{\Gamma}^{\alpha}_{\beta\gamma}=\Gamma^{\bar{\alpha}}_{\bar{\beta\gamma}}$$
and all other $\Gamma^A_{BC}$ are zero. For a complex manifold of 
complex dimension  $n$ there are  $\frac{n^2(n+1)}{2}$     distinct $\Gamma$-s.

If we take into account the hermiticity condition  in \eqref{kall}  of
the metric   and the  K\"ahlerian
restrictions \eqref{condH}, 
the non-zero Christoffel's  symbols  $\Gamma$  of the Chern connection (cf. e.g.  \cite[\S 3.2]{ball}, also Levi-Civita connection,  cf. e.g. 
 \cite[Theorem 4.17]{ball})  which appear in \eqref{geot} are determined by the equations,
see also  e.g.   \cite[(12) at p 156]{kn}
\begin{equation}\label{CRISTU}  h_{\alpha\bar{\epsilon}}\Gamma^{\alpha}_{\beta\gamma}=
\frac{\pa  h_{\bar{\epsilon}\beta}}{\pa z_{\gamma}} = \frac{\pa
  h_{\beta\bar{\epsilon}}}{\pa z_{\gamma}}, ~~\alpha, \beta, \gamma,
\epsilon =1,\dots,n, 
\end{equation}
and \[
\Gamma^{\gamma}_{\alpha\beta}=\bar{h}^{\gamma\bar{\epsilon}}\frac{\pa
  h_{\beta\bar{\epsilon}}}{\pa z_{\alpha}}=
h^{\epsilon\bar{\gamma}}\frac{\pa h_{\beta\bar{\epsilon}}}{\pa  z_{\alpha}},\quad
\text{where~ ~ ~ }
h_{\alpha\bar{\epsilon}}h^{\epsilon \bar{\beta}}=\delta_{\alpha\beta}.
\]

It is easy to prove
\begin{Remark}\label{RM10}Let $M$ be a \Ka~ manifold with local complex  coordinates
   $(z^1,\dots,z^n)$. Let
   $\Gamma^i_{jk}(z)$ be the holomorphic Christofell's symbols in the
   formula \eqref{GEOO}  of
   geodesics 
\begin{equation}\label{geoD1} 
\frac{\dd ^2 z^i}{\dd t^2}+\Gamma^i_{jk}\frac{\dd z^j}{\dd t}\frac{\dd
z^k}{\dd t} =0. \quad i=1,\dots,n.
\end{equation}
Let us make in formula
   \eqref{geoD1} the change of variables $z^j=\xi^j+\ii \eta^j$,
  and let us introduce the notation $\xi^{j'}:=\eta^j$, $j':=j+n$,
  $j=1,\dots,n$.

  Then the geodesic equations \eqref{geoD1} in
  $(z^1,\dots,z^n)\in\C^n$ became geodesic equations in the
  variables $(\xi^1,\dots,\xi^n,\xi^{1'},\dots,\xi^{n'})\in\R^{2n}$

\begin{gather*}
 \frac{\dd^2\xi^i}{\dd t^2}+\GM^i_{jk}\frac{\dd\xi^j}{\dd
  t}\frac{\dd\xi^k}{\dd t}+2\GM^i_{jk'}\frac{\dd\xi^j}{\dd
  t}\frac{\dd\xi^{k'}}{\dd t}+\GM^i_{j'k'}\frac{\dd\xi^{j'}}{\dd
  t}\frac{\dd\xi^{k'}}{\dd t} =0,\\
 \frac{\dd^2\xi^{i'}}{\dd t^2}+\GM^{i'}_{jk}\frac{\dd\xi^j}{\dd
  t}\frac{\dd\xi^k}{\dd t}+2\GM^{i'}_{j'k}\frac{\dd\xi^{j'}}{\dd
  t}\frac{\dd\xi^{k}}{\dd t}+\GM^{i'}_{j'k'}\frac{\dd\xi^{j'}}{\dd
  t}\frac{\dd\xi^{k'}}{\dd t} =0,
\end{gather*}
where 
\[
\GM^{i}_{jk}=\GM^{i'}_{j'k}=-\GM^{i}_{j'k'}=\Re{\Gamma^i_{jk}};
~-\GM^{i}_{jk'}=\GM^{i'}_{jk}=-\GM^{i'}_{j'k'}=\Im{\Gamma^i_{jk}},
\]
and the real and imaginary parts of $\Gamma^i_{jk}$ are functions of $(\xi,\xi')\in\R^{2n}$.\end{Remark}

\subsection{The Jacobi group $G^J_1(\R)$ embedded in \Spp}

We have adopted  in  \cite{SB20,SB19} the notation from  \cite{bs,ez} for the real  Jacobi group   $G^J_1(\R)$,  
realized as  submatrices of $\text{Sp}(2,\R)$ of the form
\begin{equation}\label{SP2R}
G^J_1(\R) \ni g=\left(\begin{array}{cccc} a& 0&b &   q\\
\lambda &1& \mu & \kappa\\c & 0& d &  -p\\
         0& 0& 0& 1\end{array}\right),~ M=
    \left(\begin{array}{cc}a&b\\c&d\end{array}\right),~ \det M
   =1, \end{equation}
where
\[
Y:=(p,q)=XM^{-1}=(\lambda,\mu) \left(\begin{array}{cc}a&b\\c&d\end{array}\right)^{-1}=(\lambda d-\mu
  c,-\lambda b+\mu a)
\]
is related to the Heisenberg group $\rm{H}_1$ described by the
coordinates 
$(\lambda,\mu,\kappa)$.  For  coordinatization of  the  real Jacobi
group  we use the so called  $S$-coordinates
  $(x,y,\theta,p,q,\kappa)$  \cite[Section~1.4]{bs}.

  $M\in \SL$ is realized as an element in ${\rm Sp}(2,\R)$ by the relation
\begin{gather}\label{ALOS}
M=
\left(\begin{matrix}a&b\\c&d\end{matrix}\right) \rightarrow
g = \left( \begin{matrix} a& 0&b
 &0\\0&1&0&0\\c&0&d&0\\0&0&0&1\end{matrix} \right)\!\in G^J_1(\R),
\qquad g^{-1} = \left(\begin{matrix} d& 0&-b
 &0\\0&1&0&0\\-c&0&a&0\\0&0&0&1\end{matrix}\right). \end{gather}

The Iwasawa decomposition $M=NAK$ of $M$ as in~\eqref{ALOS}
reads, see \cite[p 10]{bs}
\begin{gather}\label{MNAK}M=\left(\begin{matrix}1&x\\0&
 1\end{matrix}\right)
\left(\begin{matrix}y^{\frac{1}{2}}&
 0\\0& y^{-\frac{1}{2}}
 \end{matrix}\right)
\left(\begin{matrix}
\cos\theta &\sin\theta\\-\sin\theta
 &\cos\theta \end{matrix}\right),\qquad y>0.\end{gather}
Comparing \eqref{MNAK} with \eqref{ALOS}, we find, see \cite[p 4]{SB19}

\begin{gather*}
a = y^{1/2}\cos\theta - xy^{-1/2}\sin\theta,\qquad b = y^{1/2}\sin\theta + xy^{-1/2}\cos\theta,\\
c = -y^{-1/2}\sin\theta,\qquad d = y^{-1/2}\cos\theta,
\end{gather*}
and
\begin{gather*}
x=\frac{ac+bd}{d^2+c^2}, \qquad y=\frac{1}{d^2+c^2},\qquad \sin\theta =-\frac{c}{\sqrt{c^2+d^2}},\qquad \cos\theta =\frac{d}{\sqrt{c^2+d^2}}.
\end{gather*}

\subsubsection{The Heisenberg group embedded in  $\rm{Sp}(2,\R)$}\label{HS1}
In this section, extracted from \cite{SB20,SB19}, we
summarize  the
parametrization of the Heisenberg group used in \cite{bs}.

The composition law of the 3-dimensional  Heisenberg group $\rm{H}_1$ in \eqref{SP2R} is
\[
(\lambda,\mu,\kappa)(\lambda',\mu',\kappa')=(\lambda+\lambda',\mu+\mu',\kappa+\kappa'+\lambda\mu'-\lambda'\mu).
\]

It is easy to observe that 
 \begin{Remark}\label{MPQ1}
 For an element of $\rm{H}_1$ as element of $G^J_1(\R)$ we have in \eqref{SP2R}
\begin{equation}\label{MPQ}
M=\mathbb{1}_2 ,\quad(p,q)=(\lambda,\mu),\quad (x,y,\theta)=(0,1,0),\end{equation}
and  we denote an element of $\rm{H}_1$ embedded in $\Spp$ by
\[
\rm{H}_1\ni g =  \left(\begin{array}{cccc}
 1& 0&  0 &   \mu\\
\lambda &1& \mu & \kappa\\
0 & 0& 1 &  -\lambda\\
         0& 0& 0& 1
\end{array}\right),~ g^{-1} =  \left(\begin{array}{cccc} 1& 0& 0&   -\mu\\
-\lambda &1& -\mu & -\kappa\\
0 & 0& 1 &  \lambda\\
         0& 0& 0& 1\end{array}\right).
\]
\end{Remark}

In \cite[p 7]{SB19} we determined   the left-invariant one-forms 
\[
\lambda^p = {\rm d} \lambda,\quad
\lambda^q = {\rm d} {\mu},\quad
\lambda^r = {\rm d} {\kappa}- \lambda{\rm d} {\mu} +{\mu}{\rm d}
\lambda,
\]
and the  left-invariant vector fields on $\rm{H}_1$
\begin{equation}\label{LEFT11}
L^p=\pa _{\lambda}
                                -{\mu}\pa_ {\kappa},\quad
                                L^q= \pa_{\mu}+\lambda \pa_{\kappa},\quad  L^r=\pa_{\kappa}.
                              \end{equation}

                              The left-invariant action of $\rm{H}_1$ on itself is given by
                              \[
\exp ({\lambda}\mc{P}+{\mu}\mc{Q}+{\kappa}\mc{R})({\lambda}_0,{\mu}_0,{\kappa}_0) =
({\lambda}+{\lambda}_0,{\mu}+{\mu}_0,{\kappa}+{\kappa}_0+{\lambda}{\mu}_0-{\mu}{\lambda}_0),
\]
where the generators $\mc{P},\mc{Q},\mc{R}$ of $\rm{H}_1$ verify the commutation relations
\cite[(3.2) in the first reference]{SB19}
\[ [\mc{P},\mc{Q}]=2\mc{R},\quad [\mc{P},\mc{R}]=[\mc{Q},\mc{R}]=0.
  \]

We get the three-parameter  left-invariant metric on $\rm{H}_1$
\begin{subequations}\label{MTRSINV}
\begin{align}
g^L_{H_1}({\lambda},{\mu},{\kappa}) & =a_1(\lambda^p)^2+
a_2(\lambda^q)^2+a_3(\lambda^r)^2\\ 
& = a_1{\rm d} {\lambda}^2+a_2{\rm d}{\mu}^2 +
a_3({\rm d} {\kappa}-{\lambda}{\rm d} {\mu} +{\mu}{\rm d}
{\lambda})^2,\quad a_1,a_2,a_3>0.
\end{align}
\end{subequations}

\subsubsection{Metrics}

In \cite[first reference, Proposition 5.4, Proposition 5.6, Theorem
5.7]{SB19} we have introduced 6 invariant one-forms $\lambda^1,\dots,\lambda^6$ associated
 with the real Jacobi group 
\cite[first  reference, (4.10), (5.15), (5.17)]{SB19} and we have expressed the
invariant metric  on several  homogenous spaces associated with
$G^J_1(\R)$     in the S-coordinates 
\begin{Proposition}\label{BIGTH}
The four-parameter left-invariant metric on the real Jacobi group\\ $G^J_1(\R)$ in the S-coordinates $(x,y,\theta,$ $p,q,\kappa)$ is
\begin{equation}\label{MTRTOT}
  \begin{split}
{\rm d}s^2_{G^J_1(\R)} &=\sum_{i=1}^6\lambda_i^2 =\alpha\frac{{\rm d}
  x^2+{\rm d} y^2}{y^2} +\beta\left(\frac{{\rm d} x}{y}+2{\rm
    d}\theta\right)^2  \quad \alpha,\beta,\gamma, \delta > 0\\
& + \frac{\gamma}{y}\big({\rm d}
q^2+ S {\rm d} p^2+2x{\rm d} p{\rm d}q\big)+\delta({\rm
  d} \kappa-p{\rm d} q+q{\rm d} p)^2, ~~S:=x^2+y^2.
\end{split}
\end{equation}

Depending on  the parameter values $\alpha, \beta, \gamma, \delta \ge 0$,
in the metric \eqref{MTRTOT}, we have invariant metric on the
following $G^J_1\R)$-homogeneous manifolds$:$
\begin{enumerate}\itemsep=0pt
\item[$1)$] the Siegel upper half-plane $\mc{X}_1$ if $\beta,\gamma,\delta =0$,
\item[$2)$] the group $\SL$ if $\gamma,\delta=0$, $\alpha\beta\not= 0$, 
\item[$3)$] the Siegel--Jacobi half-plane $\mc{X}^J_1$ if $\beta, \delta= 0$, 
\item[$4)$] the extended Siegel--Jacobi half-plane $\tilde{\mc{X}}^J_1$ if $\beta=0$,
\item[$5)$] the Jacobi group $G^J_1$ if $\alpha\beta\gamma\delta\not=
  0$.
\item[$6)$] the Heisenberg group $H_1$ if \eqref{MPQ} is verified and
  $a_1=a_2=\gamma$, $a_3=\delta$ in \eqref{MTRSINV}.
\end{enumerate}
\end{Proposition}


\section{Geodesics on the Siegel--Jacobi disk}\label{GSJD}

In  Proposition \ref{PRFC} below,  extracted from \cite[Proposition 2.1 in the
first reference]{SB19} and \cite[Proposition 1]{SB20}  $(w,z)\in  ( \mc{D}_1,\C)$,
$(v,u)\in (\mc{X}_1,\C)$ and 
the parameters $k$ and $\nu$ come from representation theory of the
Jacobi group: $k$ indexes the positive discrete series of ${\rm
  SU}(1,1)$, $2k\in\N$, while $\nu>0$ indexes the representations of
the Heisenberg group.  See also \cite{jac1}, \cite{SB15}. Below by the 
Berndt-\Ka~  two-form we mean 
the invariant two-parameter \Ka~   two-form on  $\mc{X}^J_1$ determined 
in \cite{bern84,bern,cal3,cal}, see also \cite{SB20,SB19}.

\begin{Proposition}\label{PRFC}

  a) The \Ka~two-form on  $\mc{D}^J_1$, invariant to the action of
  $G^J_0=\rm{SU}(1,1)\ltimes\C$, is 
  \begin{equation}\label{kk1}
 -\ii \omega_{\mc{D}^J_1}
(w,z)\!=\!\frac{2k}{P^2}\dd w\wedge\dd
  \bar{w}+\nu \frac{\mc{A}\wedge\bar{\mc{A}}}{P},~P:=1-|w|^2,~\mc{A}=\mc{A}(w,z):=\dd
  z+\bar{\eta}\dd w.
\end{equation}

We have the change of variables $FC: (w,z)\rightarrow (w,\eta)$
\begin{gather}\label{E32}
{\rm FC}\colon \
 z=\eta-w\bar{\eta},\qquad {\rm FC}^{-1}\colon \
 \eta=\frac{z+\bar{z}w}{P},
\end{gather}
and
\[
  {\rm FC}\colon \ \mc{A}(w,z)\rightarrow \dd \eta -w\dd
  \bar{\eta}.
\]

     The matrix corresponding  to the balanced metric associated with the   \Ka~two-form \eqref{kk1}
reads 
     \begin{equation}\label{kma}
       h(w,z)
       =\left(\begin{array}{cc}h_{z\bar{z}}&h_{z\bar{w}}\\\bar{h}_{z\bar{w}}&h_{w\bar{w}}\end{array}\right)
       =\left(\begin{array}{cc}\frac{\nu}{P}&\nu\frac{\eta}{P}\\\nu\frac{\bar{\eta}}{P}&
     \frac{2k}{P^2}+\nu\frac{|\eta|^2}{P}\end{array}\right).
       \end{equation}

 b) Using the partial Cayley transform  $\Phi~:(w,z)\rightarrow (v,u)$ and its
inverse 
\begin{subequations}\label{210}
\begin{gather}
\Phi:  ~w=\frac{v-\ii}{v+\ii},~~z=2\ii
        \frac{u}{v+\ii},~~v,u\in\C,~\Im v>0,\label{210b}\\
\Phi^{-1}: ~v=\ii \frac{1+w}{1-w},~~u=\frac{z}{1-w}, ~~w,z\in\C,~ |w|<1,\label{210a}
\end{gather}
\end{subequations}
we obtain
\begin{equation}\label{ALEFT}
\mc{A}\left(\frac{v - \ii}{v+ \ii},\frac{2\ii
      u}{v + \ii}\right)=\frac{2\ii}{v+\ii}\mc{B}(v,u),
\end{equation}
where \begin{equation}\label{BFR2}
  \mc{B}(v,u) := {\rm d} u - r{\rm d} v, ~ r:=\frac{u-\bar{u}}{v-\bar{v}}.\end{equation}
The \Ka~ two-form of Berndt--\Ka, invariant to the action of
$G^J(\R)_0=\rm{SL}(2,\R)\ltimes\C$, is 
\begin{equation}
- \ii \omega_{\mc{X}^J_1}(v,u) = -\frac{2k}{(\bar{v} - v)^2} \dd
v\wedge \dd\bar{v}+ \frac{2\nu}{\ii(\bar{v} - v)}\mc{B}\wedge\bar{\mc{B}}. \label{BFR}\\
\end{equation}

 We have the   change of variables ${\rm FC}_1\colon \
 (v,u)\rightarrow (v,\eta)$
 \begin{equation}\label{FC1MIN}
   {\rm FC}_1\colon \ 2\ii u=(v+\ii)\eta-(v-\ii)\bar{\eta}, \qquad
   {\rm FC}^{-1}_1 \colon  \eta=\frac{u\bar{v}-\bar{u}v}{\bar{v}-v} +
   \ii r.\end{equation}

     The matrix corresponding  to the balanced  metric associated with the   \Ka~two-form \eqref{BFR}
reads \cite[(5.11)]{SB14}
     \begin{equation}\label{kmb}
       h(u,v)
       =\left(\begin{array}{cc}h_{u\bar{u}}&h_{u\bar{v}}\\\bar{h}_{u\bar{v}}&h_{v\bar{v}}\end{array}\right)
       =\left(\begin{array}{cc}\frac{\nu}{y}&-\nu\frac{r}{y}\\-\nu\frac{r}{y} &
     \frac{k}{2y^2}+\nu\frac{r^2}{y}\end{array}\right),~y:=\frac{v-\bar{v}}{2\ii}.
       \end{equation}

 c) If we apply the change of coordinates $\mc{D}^J_1\ni
 (v,u)\rightarrow 
 (x,y,p,q)\in\mc{X}^J_1$ \begin{equation}\label{UVPQ}
\C\ni   u:=pv+q,~~p,q\in\R, \quad \C\ni v:=x+\ii y, ~x,y\in\R,~ y>0,\end{equation}
then
\[
  r=p,
\]
\[
 \mc{B}(v,u)=\dd u -p \dd v,
\]
\begin{equation}\label{BUVpq1}
  \mc{B}(v,u)=\mc{B}(x,y,p,q):=F \dd t = v\dd p+\dd q=(x+\ii y)\dd p
  +\dd q, \quad  F:= \dot{p}v+\dot{q}.
\end{equation}

d)   Given  \eqref{UVPQ} and
\[
  \C\ni u:=\xi+\ii \rho, \quad \xi,\rho\in\R,
\]
 we obtain  the change of variables
\begin{equation}\label{LASt1}
 (x,y,\xi,\rho)\rightarrow (x,y,p,q): ~\xi=px+q,~\rho=py,
\end{equation}
and 
\[
  \mc{B}(v,u)=\mc{B} (x,y,\xi,\rho) =  \dd u-\frac{\rho}{y}\dd v=\dd (\xi+\ii \rho)-\frac{\rho}{y}\dd (x+\ii y).
 \]
\end{Proposition}
Based on \cite[Remark 1]{SB20},  we complete Proposition \ref{PRFC} with  the relation between 
the parameter $\eta$ introduced in
\eqref{E32} and the S-variables $p,q$,   and we write  down explicitly the change of
variables   $(w,z)\rightarrow (v,\eta)$ 
\begin{Proposition}\label{STR}
  The FC-transform \eqref{E32} relates Perelomov's  un-normalized CS-vector
  $e_{wz}$ with the normalized one $\underline{e}_{w\eta}$
 \[
    \underline{e}_{w\eta }=(e_{wz}, e_{wz})^{-\frac{1}{2}}e_{wz},\quad
    w\in\mc{D}_1,~z,\eta \in \C,
  \]
  and the S-variables $p,q$ introduced by \eqref{UVPQ} are related to the parameter $\eta$
  \eqref{E32} by the
simple relation
\begin{equation}\label{strange}
    \eta=q +\ii p.
  \end{equation}
The second partial Cayley transform 
$$\Phi_1:=FC_1\circ\Phi:  (w,z)\rightarrow
(v=x+\ii y,\eta=q+\ii p)$$
  from the Siegel-Jacobi disk
$\mc{D}^J_1$ to the Siegel-Jacobi upper half-plane $\mc{X}^J_1$
(respectively its inverse) is  given 
 by  \eqref{ULTRAN1}  (respectively,   \eqref{ULTRAN2})
\begin{subequations}\label{ULTRAN}
\begin{gather}
 \Phi_1: w=\frac{v-\ii}{v+\ii}, \quad  z=2\ii \frac{pv+q}{v+\ii},\label{ULTRAN1}\\
\Phi_1 ^{-1}: v=\ii \frac{1+w}{1-w}, \quad \eta =
\frac{(1+\ii \bar{v})(z-\bar{z})+v(\bar{v}-\ii)(z+\bar{z})}{2\ii
  (\bar{v}-v)}=\frac{z+\bar{z}w}{P}.\label{ULTRAN2}
\end{gather}
\end{subequations}
 \end{Proposition}

\begin{proof}
Relation \eqref{strange} was  proved in \cite[Remark 1]{SB20}, based
on  \cite[Comment
6.12]{jac1},  (2.1),
(2.2) and  \cite[Lemma 2]{SB14}. Here  we give another very simple proof.

Equating the values of $u$ in  \eqref{UVPQ}  and  \eqref{FC1MIN}, 
 we get the formula
\begin{equation}\label{PQVV}
v(\eta-\bar{\eta}-2\ii p)+\ii(\eta+\bar{\eta}-2q)=0,
\end{equation}
If in \eqref{PQVV} we recall that in  $v=x+\ii y$ we have  $y>0$,
then 
\eqref{strange} follows.

 If in the second formula \eqref{210a}  it is
introduced first formula \eqref{210b},  
\eqref{ULTRAN1} follows.

The proof of  formula \eqref{ULTRAN2} is also very easy.
\end{proof}

We emphasize that the new relations \eqref{ULTRAN} are a more precise 
formulation of relations of the change of coordinates which appear in
\cite[Remark 8]{SB14} and
  \cite[Proposition 6]{GAB}.

We also  make the  
\begin{Remark}\label{REMRC2}
If in formula \eqref{kk1} of $\mc{A}(z,w)$ we
introduce  the  differential  of $z$ given in \eqref{ULTRAN1} we get 
\begin{equation}\label{DOTZ}
\dot{z}=\frac{2\ii}{v+\ii}\left(F-\frac{\bar{\eta}\dot{v}}{v+\ii}\right).
\end{equation}
With \eqref{strange} and  \eqref{ALEFT}  equation \eqref{BUVpq1} is proved.
\end{Remark}

 Proposition \ref{GER} below is extracted
from \cite[Remark 7.3]{jac1}, \cite[Remark 3.3]{FC}, \cite[Remark 7]{SB14},
\cite[Proposition 1]{GAB}.
\begin{Proposition}\label{GER}
  The geodesic equations on the Siegel-Jacobi disk in the
 $(w,z)$-variables corresponding to the 
metric defined by the \Ka~two form  \eqref{kk1} are 
\begin{subequations}\label{geo}
\begin{align}
\bar{\eta}G^2_1 & = 2\iota G_3, \quad G_1:=\frac{\dd z}{\dd
  t}+\bar{\eta}\frac{\dd w}{\dd t}, ~ G_3 := \frac{\dd ^2z}{\dd
  t^2}+2\frac{\bar{w}}{P}\frac{\dd z}{\dd t}\frac{\dd w}{\dd t},  \quad \iota:=\frac{\kappa}{\nu};  \label{geo1}\\
G^2_1 & = -2\iota G_2, \quad G_2:=\frac{\dd^2w}{\dd
  t^2}+2\frac{\bar{w}}{P}(\frac{\dd w}{\dd t})^2 . \label{geo2}
\end{align}
\end{subequations}
\end{Proposition}

A particular solution of \eqref{geo} is  given in
\begin{Remark}\label{RE2} With $G_2$ defined in \eqref{geo2},
the  geodesic equation $G_2=0$  on $\mc{D}_1$ has the solution
\begin{equation}\label{WT}
    w(t) =\frac{B}{|B|}\tanh{|B|t},\quad w(0)=0,~~\dot{w}(0)=B.
  \end{equation}
 $\eta=\eta_0=ct$ gives  a particular solution of equations of
  geodesics \eqref{geo} on $\mc{D}^J_1$
\[
    w(t), \quad z(t)=\eta_0-\bar{\eta_0}w(t).
\]
\end{Remark}
\begin{proof}
  In \cite{ber97A} we proved that a $w(t)$
  of the type \eqref{WT}
  is a solution of $G_2=0$ for a complex Grassmann manifold
  $G_n(\C^{n+m})$ and its noncompact dual. 
  If $G_2=0$, then $G_1=0$ and  also
  \begin{equation}\label{EQ1}
    \dot{z}=-\bar{\eta}\dot{w}.
      \end{equation}
      Differentiating $z$ in \eqref{E32}, we find with \eqref{EQ1}
      $$ \dot{\eta}-w\dot{\bar{\eta}}=0, $$ which combined with its 
      complex conjugate implies the condition $$P\dot{\eta}=0,$$
    i. e. $\eta=\eta_0$.

    Conversely,  if $\eta$ is constant, then $G_1=0$, which also
    implies  $G_2=0$ and also  $G_3=0$.
  \end{proof}

  Now we pass from   the geodesic equations \eqref{geo} on $\mc{D}^J_1$ in the
  complex variables $(w,z)$ to  the real variables $(m,n,\alpha,\beta)$, where
  \begin{equation}
    \label{mnab}
   \C\ni w=\alpha  +\ii \beta, ~|w|<1,\quad \C\ni z=m+\ii n, ~\alpha,
   \beta, m,n\in \R.
 \end{equation}
 If we introduce \eqref{mnab} into the complex variable $\eta$  defined in
 \eqref{E32}, we get
 \begin{equation}\label{PMN}
   P\eta=C+\ii D, \text{~where~}
   C:=(1+\alpha)m+n\beta,~~D:=(1-\alpha)n+m\beta.
 \end{equation}
 We introduce also the notation
 \begin{subequations}\label{TOATE}
   \begin{align}
     W &:=C^2-D^2, ~T:=2CD,~ U:=2\alpha
         +\frac{j}{P}W,~V:=\beta+\frac{jT}{2P},~j:=\frac{1}{2\iota},  \\
     Z_1& 
     := \alpha-j\frac{W}{P}.~ Z_2:=\beta-j\frac{T}{P}, ~
          X :=\frac{3C^2-D^2}{P^2},~Y:=\frac{3D^2-C^2}{P^2},\\
     L &  :=
         \dot{m}\dot{\beta}+\dot{n}\dot{\alpha}, \quad
         M:=\dot{m}\dot{\alpha}-\dot{n}\dot{\beta}.    \end{align}
   \end{subequations}
   If we  express $W,T,X,Y$ defined in \eqref{TOATE} as a function of
   the variables $m,n,\alpha,\beta$,  we get 
  \begin{subequations}\label{TOATE1}
 \begin{align}
 W & =[(\alpha+1)^2-\beta^2]m^2+[\beta^2-(\alpha-1)^2]n^2+4mn\alpha\beta,\\
 \frac{T}{2} &=\beta[(\alpha+1)m^2+(-\alpha+1)n^2]+mn(1-\alpha^2+\beta^2),\\
P X & = [3(\alpha+1)^2-\beta^2]m^2+[3\beta^2-(\alpha-1)^2]n^2+4(2\alpha+1)\beta mn,\\
P Y & =[3\beta^2-(\alpha+1)^2]m^2+[3(\alpha-1)^2-\beta^2]n^2+4(1-2\alpha)\beta mn.
\end{align}
\end{subequations}
   We now prove
   \begin{Proposition}\label{PP55}
     In the notation \eqref{PMN}, \eqref{TOATE}, \eqref{TOATE1},   the geodesic
     equations \eqref{geo} on the Siegel-Jacobi disk
corresponding to the \Ka~ two-form \eqref{kk1}
 expressed in the real variables \eqref{mnab} are
\begin{subequations}\label{ECMNAB}
   \begin{align}
     &\ddot{\alpha}+\frac{U}{P}(\dot{\alpha}^2-\dot{\beta}^2)+j(\dot{m}^2-\dot{n}^2)+
       \frac{4V}{P}\dot{\alpha}\dot{\beta}+\frac{2j}{P}(DL+CM)=0,\\
     &\ddot{\beta}-\frac{2V}{P}(\dot{\alpha}^2-\dot{\beta}^2)+\frac{2U}{P}\dot{\alpha}\dot{\beta}+2j(\dot{m}\dot{n}
       +\frac{W}{P})=0,\\
 & \ddot{m}\!+\!\frac{jC}{P}[\frac{Y}{P^2}(\dot{\alpha}^2\!-\!\dot{\beta}^2)
   \!-\!\dot{m}^2\!+\!\dot{n}^2]
   \!+\!\frac{2jD}{P}(-\frac{X}{P^2}\dot{\alpha}\dot{\beta} \!+\!\dot{m}\dot{n})
       \!+\!\frac{2}{P}(Z_1M\!+\!Z_2L)=0,\\
     & \ddot{n}\!+\!\frac{jD}{P}[(\frac{X}{P^2}(\dot{\alpha}^2\!-\!\dot{\beta}^2)\!+\!
      \dot{m}^2\!-\!\dot{n}^2]\!+\!2j\frac{CY}{P^3}\dot{\alpha}\dot{\beta}\!+\!\frac{2}{P}(-jC\dot{m}\dot{n}\!+\!Z_1L\!+\!Z_2M)=0.
   \end{align}
   \end{subequations} 
 \end{Proposition}\label{NUMER}
 \begin{proof}
   The balanced metric associated to  the \Ka~ two-form \eqref{kk1}
   on $\mc{D}^J_1$ reads
   $$\dd s^2_{\mc{D}^J_1}(w,z)=h_{z\bar{z}}\dd z\dd
   \bar{z}+h_{w\bar{w}}\dd w\dd \bar{w}+ h_{z\bar{w}}\dd z\dd
   \bar{w}+h_{{\bar z}w} \dd \bar{z}\dd w,$$
   where  the matrix $h$ is defined in \eqref{kmb}. With the notation
   \eqref{mnab}, the metric matrix attached to the metric $\dd
  s^2_{\mc{D}^J_1}(m,n,\alpha,\beta)$ is
\begin{equation}\label{mMNAB}
g_{{\mc{D}}^J_1}(m,n,\alpha,\beta)\! = \!\left(\begin{array}{cccc}g_{mm} &0 &g_{m\alpha} &g_{m\beta}\\
0& g_{nn}& g_{n\alpha}& g_{n\beta} \\
g_{\alpha m}&  g_{\alpha n} & g_{\alpha\alpha} & 0\\g_{\beta m}& 
                                                                    g_{\beta
                                                                 n}& 0
                                                                       & g_{\beta\beta}

 \end{array}\right),\!\quad 
 \begin{array}{ll}g_{mm}\!&=g_{nn}=\frac{\nu}{P} \\
 \!g_{\alpha\alpha}&=g_{\beta\beta}=\frac{2k}{P^2}+\frac{\nu}{P^3}(C^2+D^2)\\
g_{m\alpha}\!&=\!g_{n\beta}=\frac{\nu}{P^2}C\\
g_{\beta m} &=-g_{\alpha n}= \frac{\nu}{P^2}D
\end{array}.
\end{equation}
Now we calculate the determinant of the metric matrix \eqref{mMNAB}.
 We recall, cf. \cite[Lemma 1, p 176]{Hua2}: \\
Let  us consider the  analytic mappings
\begin{equation}\label{EQ112}
  w_k:=f_k(z_1,\dots,z_k), \quad z_k=x_k+\ii
  y_k,\quad w_k=u_k+\ii v_k, \quad  k=1,\dots,n.
\end{equation}
 The mapping \eqref{EQ112} induces a transformation of the
 $2n$-dimensional real space. We have
 \begin{equation}\label{2det}
   \frac{\pa (w_1,\dots,w_n, \bar{w}_1,\dots,\bar{w}_n)}{\pa
     (z_1,\dots,z_n, \bar{z}_1,\dots,\bar{z}_n)}=|\frac{\pa w}{\pa
     z}|^2,\end{equation}
   where  $$\dd w_i= \sum_{i=1}^n\dd z_i \frac{\pa f_i}{\pa z_j},\quad
    \dd \bar{w}_i= \sum_{i=1}^n\bar{\dd z_i }\bar{\frac{\pa f_i}{\pa z_j}}.$$

From the relation
\begin{equation}\label{deth}
  \det h(w,z)=\frac{2 k \nu}{P^3}
  \end{equation}
 with \eqref{2det} we get 
\begin{equation}\label{DETH2}
  \det g (m,n,\alpha,\beta)=(\frac{2 k \nu}{P^3})^2.
\end{equation}

We make use  of the change of variables \eqref{mnab}, \eqref{PMN} in the
equations \eqref{geo} and we express $G_i$ as
\begin{equation}\label{SUBST}
  \C\ni G_i:=A_i+\ii
B_i,~A_i, B_i\in \R,~~ i=1,2,3,\end{equation} where
\begin{equation}\label{A1B1}
\begin{array}{lll}
 A_1 & := \ddot{m}+ \frac{1}{P}(C\dot{\alpha}+D\dot{\beta})  & B_1:
 =\ddot{n}+\frac{1}{P}(C\dot{\beta}+D\dot{\alpha})\\
 A_2&:=\ddot{\alpha}+\frac{2}{P}[\alpha(\dot{\alpha}^2-\dot{\beta}^2)+2\beta
 \dot{\alpha}\dot{\beta}] & 
 B_2 :=\ddot{\beta}+\frac{2}{P}[2\alpha\dot{\alpha}\dot{\beta}+\beta(-\dot{\alpha}^2+\dot{\beta}^2)]   \\
  A_3 &:=\ddot{m}+\frac{2}{P}(\alpha M+\beta  L) &
                                                   B_3 :=\ddot{n}+\frac{2}{P}(\alpha  L-\beta M)\end{array}.
\end{equation}
Introducing the expressions \eqref{SUBST} into \eqref{geo},
the geodesic equations  become 
\begin{subequations}\label{A2B2}
  \begin{align}
    A_2&+j(A_1^2-B_i^2)=0, \\ B_2 & +2jA_1B_1=0,\\
    A_3& -2j\Im\eta A_1B_1-j\Re\eta(A_1^2-B_1^2)=0,\\
    B_3& -2j\Re\eta  A_1B_1+j\Im\eta(A_1^2-B_1^2)=0.
  \end{align}
  \end{subequations}
  Introducing \eqref{A1B1} in \eqref{A2B2}, with
  \begin{gather*}
   A^2_1-B^2_1 
                   =\dot{m}^2-\dot{n}^2+\frac{W}{P^2}(\dot{\alpha}^2-\dot{\beta}^2)+2\frac{T}{P^2}\dot{\alpha}\dot{\beta}+\frac{2}{P}(CM+DL),\\
        A_1B_1  =
                 -\frac{1}{2P^2}(\dot{\alpha}^2-\dot{\beta}^2)+\dot{m}\dot{n}+\frac{W}{P}\dot{\alpha}\dot{\beta}
                 +\frac{1}{P}(CL-DM),
 \end{gather*}    
  we get \eqref{ECMNAB}.
\end{proof}

 \section{Geodesics on the Siegel--Jacobi upper half-plane}\label{GSJUH}

  The geodesic equations on $\mc{X}^J_1$ expressed in the variables
  $(v,u)$ are given in
\begin{Proposition}\label{P3}
 a)   The differential equations  on $\mc{X}^J_1$ obtained   by the
 partial Cayley transform \eqref{ULTRAN}  from the 
  geodesic equations  on $\mc{D}^J_1$ \eqref{geo} are
\begin{subequations}\label{geoX}
\begin{align}
rH^2_1 & = \iota H_3, \quad
                       H_1:=\dot{u}-r\dot{v},
                       ~ H_3 :=  \ii\ddot{u}-\frac{\dot{u}\dot{v}}{y};
                     \label{geoX1}\\
 H^2_1 & = \iota H_2, \quad H_2:= \ii
            V_1=\ii\ddot{v}-\frac{\dot{v}^2}{y}, \quad V_1:=
            \ddot{v}-\frac{\dot{v}^2}{\ii y}=\ddot{v}+2\frac{\dot{v}^2}{\bar{v}-v}, \label{geoX2}
\end{align}
\end{subequations}
where  $r$ was defined in \eqref{BFR2}.
Equations \eqref{geoX} can be written as  
 \begin{subequations}\label{ECIV_1}
\begin{align}
& \ddot{v}+\frac{\ii}{\iota}\left[\dot{u}^2-2r\dot{u}\dot{v}+
(\frac{\iota}{y}+r^2)\dot{v}^2\right]=0,\\
& \ddot{u}+\frac{\ii}{\iota}\left[
  -r\dot{u}^2+\frac{\iota}{y}(1-2r^2)\dot{u}\dot{v}+r^3\dot{v}^2
 \right]=0.
\end{align}                                             
\end{subequations}

b) The differential  equations \eqref{geoX} or  \eqref{ECIV_1} are geodesic equations  in the variables
 $(v,u)$ on $\mc{X}^J_1$
 attached to the metric corresponding to the \Ka~ two-form \eqref{BFR}.

c) The partial Cayley transform  $\Phi: (w,z)\rightarrow (v,u)$ given
in \eqref{210}
is a geodesic mapping   $\mc{D}^J_1\rightarrow \mc{X}^J_1$.

\end{Proposition}

\begin{proof}
 
 {\bf a)} In \eqref{geo} we make  the change of
coordinates 
$(w,z)\rightarrow (v,u)$ given by
the partial Cayley transform \eqref{210}.

  With \eqref{210b}, we get successively
\begin{subequations}\label{32}
\begin{align}
P &=\frac{4y}{(v+\ii)(\bar{v}-\ii)};\label{32a}\\\dot{w} &=\!2\ii \frac{\dot{v}}{(v+\ii)^2};\label{32b}\\
 2\frac{\bar{w}}{P}\ddot{w}^2&=-2\frac{\bar{v}+\ii}{y(v+\ii)}[(v+\ii)\bar{u}-u\dot{v}]\dot{v};\label{32c}\\
\ddot{w} &  =2\ii\frac{(v+\ii)\ddot{v}-2\dot{v}^2}{(v+\ii)^3}; \label{32d}\\
 \dot{z} &=\frac{2\ii}{(v+\ii)^2} [(v+\ii)\dot{u}-u\dot{v}];\label{32e}\\
  \ddot{z} & =\frac{2\ii}{(v+\ii)^3}[(v+\ii)^2\ddot{u}-u(v+\ii)\ddot{v}-
 2(v+\ii)\dot{v}\dot{u}+2u\dot{v}^2].\label{32f}
\end{align}
\end{subequations}
  In the expression  \eqref{geo2} of $G_2(w,z)$ we introduce \eqref{32b},
\eqref{32c}, \eqref{32d} and we obtain
  \begin{equation}\label{G2}G_2(v,u)=\frac{2\ii V_1}{(v+\ii)^2}.\end{equation}
  In the expression \eqref{geo1} of $G_1(w,z)$ we introduce
  \eqref{32b} and we get
  $$G_1(v,u)=\frac{2}{(v+\ii)^2}[(v+\ii)\dot{u}-(\bar{\eta}-u)\dot{v}].$$
  But  expression \eqref{FC1MIN} of $\eta$ implies
  \begin{equation}
    \label{BARE}\bar{\eta}-u=-r(v+\ii),\end{equation}
  and we find 
  \begin{equation}\label{G1}G_1(v,u)=\frac{2\ii}{v+\ii}(\dot{u}-r\dot{v}).\end{equation}
  If we introduce \eqref{G1} and \eqref{G2}  in \eqref{geo2} we get
  \eqref{geoX2}.

  To  also prove \eqref{geoX1},  we introduce 
  in  expression \eqref{geo1} of $G_3(w,z)$  relations \eqref{32a}, 
  \eqref{32b}, \eqref{32e}, \eqref{32f} and   we get
$$G_3(v,u)=\frac{2}{(v+\ii)^3}\{\ii[(v+\ii)^2\ddot{u}-2(v+\ii)\dot{v}\dot{u}+(2\dot{v}^2-u(v+\ii)\ddot{v})]-(v+\ii)(\bar{v}+\ii)\dot{u}\dot{v}+(\bar{v}+\ii)\frac{u\dot{v}^2}{y}\}.$$
But $$2\ii u\dot{v}^2+
u\dot{v}^2\frac{\bar{v}+\ii}{y}=-(v+\ii)\frac{\dot{u}\dot{v}}{y};\quad
(-2\ii-\frac{\bar{v}+\ii}{y})\dot{u}\dot{v}=-(v+\ii)\frac{\dot{u}\dot{v}}{y},$$
and we get for $G_3(v,u)$ 
\begin{equation}\label{G3}G_3(v,u)=\frac{2}{v+\ii}(\ii
  \ddot{u}-\frac{\dot{u}\dot{v}}{y})-uG_2.\end{equation}

 Now the expressions \eqref{G1},\eqref{G2} and \eqref{G3} are
  introduced in \eqref{geo} and we get
  \begin{subequations}\label{geoY}
\begin{align}
 \bar{\eta}H_1^2 & = -\iota [(v+\ii)H_3-uH_2],\label{geoY1}\\
     H_1^2 & = \iota H_2. \label{geoY2} 
\end{align}
\end{subequations}

If we introduce \eqref{geoY2} into \eqref{geoY1}, we have
\[
  (\bar{\eta}-u)H^2_1=-(v+\\i) H_3,
\]
and, with \eqref{BARE},  we get \eqref{geoX1}.

 {\bf b)} We calculate the geodesic equations on the Siegel-Jacobi disk
$\mc{X}^J_1$ in the variables $(v,u)$ corresponding to the balanced
metric \eqref{kmb}.

In the variables $(u,v)\in(\C,\mc{X}_1)$ the  geodesic equations 
\eqref{GEOO}  for the metric \eqref{kmb} read
\begin{equation}\label{geomic}
 \left\{
 \begin{array}{l}
 \frac{\dd^2 u}{\dd t^2}+\Gamma^u_{uu}\left(\frac{\dd u}{\dd
     t}\right)^2 +
2\Gamma^u_{uv}\frac{\dd u}{\dd t}  \frac{\dd v}{\dd t} +\Gamma
^u_{vv}\left(\frac{\dd v}{\dd t}\right)^2 =0   ;\\
  \frac{\dd^2 v}{\dd t^2}+\Gamma^v_{uu}\left(\frac{\dd u}{\dd
     t}\right)^2 +
2\Gamma^v_{uv}\frac{\dd u}{\dd t}  \frac{\dd v}{\dd t} +\Gamma
^v_{vv}\left(\frac{\dd v}{\dd t}\right)^2=0 .
     \end{array}
 \right.
\end{equation}
The equations \eqref{CRISTU} which determine the $\Gamma$-symbols for
the Siegel-Jacobi disk are 
\begin{equation}\label{ec528}
 \left\{
 \begin{array}{l}
h_{u\bar{u}}\Gamma^u_{uu}+h_{v\bar{u}}\Gamma^v_{uu}=
\frac{\pa h_{u\bar{u}}}{\pa u}; \\
h_{u\bar{v}}\Gamma^u_{uu}+h_{v\bar{v}}\Gamma^v_{uu}=
\frac{\pa h_{u\bar{v}}}{\pa u}. 
\end{array}
 \right.
\end{equation}
\begin{equation}\label{ec529}
 \left\{
 \begin{array}{l}
h_{u\bar{u}}\Gamma^u_{uv}+h_{v\bar{u}}\Gamma^v_{uv}=
\frac{\pa h_{v \bar{u}}}{\pa u}; \\
h_{u\bar{v}}\Gamma^u_{uv}+h_{v\bar{v}}\Gamma^v_{uv}=
\frac{\pa h_{v\bar{v}}}{\pa u}. 
\end{array}
 \right.
\end{equation}
\begin{equation}\label{ec530}
\left\{
 \begin{array}{l}
h_{u\bar{u}}\Gamma^u_{vv}+h_{v\bar{u}}\Gamma^v_{vv}=
\frac{\pa h_{v \bar{u}}}{\pa v}; \\
h_{u\bar{v}}\Gamma^u_{vv}+h_{v\bar{v}}\Gamma^v_{vv}=
\frac{\pa h_{v\bar{v}}}{\pa v}. 
\end{array}
 \right.
\end{equation}

We calculate  easily the partial derivatives of the
elements of the metric \eqref{kma}
\begin{equation}\label{ec531}
\begin{split}
\frac{\pa h_{u\bar{u}}}{\pa u } &= 0,~ \frac{\pa h_{u\bar{v}}}{\pa u}=\frac{\ii\nu}{2y^2}
, ~ \frac{\pa h_{v\bar{v}}}{\pa u }=-\frac{\ii\nu r}{y^2};\\
\frac{\pa h_{v\bar{u}}}{\pa v } & = -\frac{\ii \nu r}{y^2}, \frac{\pa h_{v\bar{v}}}{\pa v}=
\frac{\ii k}{2y^3}+\frac{3}{2}\ii\nu \left(\frac{r}{y}\right)^2.
\end{split}
\end{equation}

  Introducing \eqref{ec531} into \eqref{ec528}-\eqref{ec530}, we find
for   the  Christoffel's symbols $\Gamma$-s  the following  expressions
\begin{equation}\label{GAMM}
\begin{split}
\Gamma^u_{uu}  & =-\frac{\ii}{\iota}r,~\Gamma^v_{uu}=\frac{\ii}{\iota}, ~
\Gamma^u_{uv}= \frac{\ii}{2}(\frac{1}{y}-2\frac{r^2}{y});\\
 \Gamma^v_{vu} & =-\frac{\ii}{\iota}r,~
 \Gamma^u_{vv}=\frac{\ii}{\iota}r^3, ~ \Gamma^v_{vv} =
 \ii(\frac{1}{y}+\frac{r^2}{\iota}).
\end{split}
\end{equation}
Introducing \eqref{GAMM} in \eqref{geomic}, we get the equations 
\eqref{ECIV_1} on $\mc{X}^J_1$ in the variables $(u,v)\in (\C, \mc{X}_1)$.

{\bf c)} is proved because  the equations
\eqref{geoX}  represent  geodesic equations \eqref{ECIV_1} for  
Christofell symbols \eqref{GAMM} associated to metric corresponding
to the \Ka~two-form \eqref{BFR}.
  
 \end{proof}

A particular solution of \eqref{geoX} on $\mc{X}^J_1$ in the variables
$(v,u)$ is given in 
\begin{Remark}\label{RE3}
  The  Cayley transform \eqref{210a} $\mc{D}_1\ni w\rightarrow v
  \in \mc{X}_1$ is a geodesic mapping $\Phi^{-1}: \mc{D}_1\rightarrow
  \mc{X}_1$.
  
 If $w(t)$ is the solution \eqref{WT} of the equation $G_2=0$ of
 geodesics on the
 Siegel disk $\mc{D}_1$, then   geodesic
 equations $H_2=0$ on
 the Siegel upper half-plane $\mc{X}_1$ with
 $H_3$ defined in \eqref{geoX2}
  have the solution
\begin{equation}\label{WT1}
    v(t) =\ii \frac{1+w(t)}{1-w(t)},\quad v(0)=\ii,~~\dot{v}(0)=2v(0)B.
  \end{equation}
 $\eta=\eta_0$ gives  the particular solution of equations of
  geodesics \eqref{geoX} on $\mc{X}^J_1$
  \begin{equation}\label{SPART1}
    v=v(t), \quad u(t)=\frac{\eta_0-\bar{\eta_0}w(t)}{1-w(t)},\quad u(0)=\eta_0.
  \end{equation}
\end{Remark}
\begin{proof}
  For the first assertion it is observed that
  $$H_2=-2\frac{P}{(1-w)^2Q}G_2,$$ where
 \begin{equation}\label{QQQ}
 Q:=|1-w|^2.
\end{equation}
  $G_2=0$ implies $H_2=0$ and $v$ given in \eqref{WT1} is a
  solution of $H_2=0$ if $w$ is a solution of $G_2=0$. 
  \end{proof}

Now we are intersted in the geodesic equations on the
Siegel-Jacobi upper half-plane expressed in the variables
$(v,\eta)=(x+\ii y,q+\ii p)$.
 \begin{Proposition}\label{P4}

a) The two-parameter   balanced  metric  on the
Siegel--Jacobi upper half-plane $\mc{X}^J_1$, the particular case of
\eqref{MTRTOT}  corresponding to  item 
\emph{3)} in \emph{Proposition \ref{BIGTH}},   associated to the \Ka~
two-form \eqref{BFR}, \eqref{BUVpq1}, is 
\begin{equation}\label{METRS2}
  \dd s^2_{\mc{X}^J_1}(x,y,p,q)  \!=\!
\alpha\frac{\dd x^2\!+\!\dd   y^2}{y^2} +\frac{\gamma}{y}(S\dd p^2+\dd q^2+2x\dd
  p\dd q),
\end{equation}
where $S$ was defined in \eqref{MTRTOT} and
\begin{equation}\label{AK}
  \alpha:=k/2,\quad\gamma:=\nu.
  \end{equation}
  
The geodesic equations on $\mc{X}^J_1$ corresponding to the
metric \eqref{METRS2} are
\begin{subequations}\label{420}
  \begin{align}
  E_1 &:= \ddot{x}-\frac{2}{y}\dot{x}\dot{y}-\ep R y\dot{p}=0 ,\quad
        R:=\Re F=x\dot{p}+\dot{q}, \quad \epsilon:=\frac{\gamma}{\alpha},\label{421a}\\
  E_2 &:=
        \ddot{y}+\frac{1}{y}(\dot{x}^2-\dot{y}^2)+\frac{\ep}{2}(R^2-y^2\dot{p}^2)=0,
\label{421b}\\
  E_3 &:= \ddot{p}+R\frac{\dot{x}}{y^2}+\frac{1}{y}\dot{y}\dot{p}=0,\label{421c}\\
  E_ 4&:= \ddot{q}+\frac{\dot{x}}{y^2}(y^2\dot{p}-xR)-\frac{\dot{y}}{y}(R+x\dot{p})=0 .\label{421d}
 \end{align}
\end{subequations}
If in \eqref{420} we take $\epsilon =0$ ($\gamma=0$) we get the
expression of  geodesic  equations  $H_2=0$ in \eqref{geoX2} in the variables  $(x,y)$ on $\mc{X}_1$ corresponding to the Christofell's symbols
\eqref{GM22}
\begin{subequations}\label{GM22}
  \begin{align}
    \Gamma^x_{xx} & = 0,\quad  \Gamma^x_{xy}  = -\frac{1}{y},\quad
                    \Gamma^x_{yy}  = 0,\\
     \Gamma^y_{xx} & = \frac{1}{y},\quad  \Gamma^y_{xy}  = 0,\quad  \Gamma^y_{yy}  = -\frac{1}{y}.
 \end{align}
 \end{subequations}
associated with the metric in case \emph{1)} in \emph{Proposition \ref{BIGTH}}
\[
\ddot{x} -\frac{2}{y}\dot{x}\dot{y}=0,\qquad
\ddot{y}  +\frac{1}{y}(\dot{x}^2-\dot{y}^2)=0,
\]
with the solution \eqref{WT1}.

b) With the change of coordinates  $\mc{D}^J_1\ni (v,u)\rightarrow
(x,y,p,q)\in \mc{X}^J_1$ \eqref{UVPQ},  the geodesic equations \eqref{ECIV_1} on $\mc{D}^J_1$
become on $\mc{X}^J_1$ the system of differential equations
   \begin{subequations}\label{geoYY}
\begin{align}
 pK^2_1 & = \iota K_3, \quad
                       K_1:=F,
                       ~ K_3 :=  \ii(\ddot{p}v+\ddot{q}) - \bar{F}\frac{\dot{v}}{y}+pH_2;
                     \label{geoYY1}\\
 K^2_1 & = \iota K_2, \quad K_2:= H_2, \quad v=x+\ii y,
\label{geoYY2}
\end{align}
\end{subequations}
where $F$ was defined in \eqref{BUVpq1}.

c) The transform $(v,u)\rightarrow (x,y,p,q)$ \eqref{UVPQ} is a
geodesic mapping
$\mc{D}^J_1\rightarrow \mc{X}^J_1$.

$\eta= q+\ii p = ct$ is a particular solution of \eqref{geoYY}.

\end{Proposition}
\begin{proof}
{\bf{a )}}
The matrix associated  with the   metric \eqref{METRS2} is
\begin{equation}\label{begGGm}
g_{{\mc{X}}^J_1}\! = \!\left(\begin{array}{cccc}g_{xx} &0 &0 &0\\
0& g_{yy}& 0& 0 \\
0& 0& g_{pp} & g_{pq} \\0 & 0& g_{qp}& g_{qq}
 \end{array}\right),\!
 \begin{array}{cc}\qquad g_{xx}\!=\frac{\alpha}{y^2}, &
 \!g_{yy}\!=\!\frac{\alpha}{y^2};\\
g_{pq}\!=\!\gamma\frac{x}{y} , &
g_{pp} \!=\!\gamma\frac{S}{y},\quad g_{qq}\!=\!\frac{\gamma}{y},
\end{array}
\end{equation}
and 
\begin{equation}\label{DD2}
\det(g_{{\mc{X}}^J_1}(x,y,q,p))=\left(\frac{\alpha\gamma}{y^2}\right)^2.
\end{equation}
The invers of the metric matrix \eqref{begGGm} is
\[
g^{-1}_{{\mc{X}}^J_1}\! = \!\left(\begin{array}{cccc}g^{xx} &0 &0 &0\\
0& g^{yy}& 0& 0 \\
0& 0& g^{pp} & g^{pq} \\0 & 0& g^{qp}& g^{qq}
 \end{array}\right),\!
 \begin{array}{llll}g^{xx}\!= \!g^{yy}\!=\!\frac{y^2}{\alpha}, &
   g^{pq}\!=\!-\frac{x}{\gamma y},&
 g^{pp} \!=\!\frac{1}{\gamma y},& g^{qq}\!=\!\frac{S}{\gamma y}.
 \end{array}
\]

With
 formula \eqref{geoI} we determine the non-zero  Christoffel's symbols corresponding to
 the Riemannian metric \eqref{METRS2} of the Siegel-Jacobi upper half-plane
\begin{equation}
  \begin{array}{lllll}\label{GSC}
    \Gamma^{x}_{xy}=-\frac{1}{y}
  &\Gamma^{x}_{pp}=-\ep xy
  &\Gamma^{x}_{pq}= -\frac{1}{2}\ep y & & \\
   \Gamma^y_{xx}= \frac{1}{y}&
 \Gamma^y_{yy}=-\frac{1}{y}&\Gamma^y_{pp}=\frac{\ep}{2}(x^2\!\!-\!\!y^2)&\Gamma^y_{pq}=\frac{\ep}{2}x
  &\Gamma^y_{qq}=\frac{\ep}{2} \\
  \Gamma^{p}_{xp}=\frac{1}{2}\frac{x}{y^2} &  \Gamma^{p}_{xq}=
 \frac{1}{2}\frac{1}{y^2}&\Gamma^{p}_{yp}= \frac{1}{2y} & & \\
  \Gamma^q_{xp}=\frac{y^2-x^2}{2y^2} &\Gamma^q_{xq}= -\frac{x}{2y^2}&\Gamma^q_{yp}=-\frac{x}{y}&\Gamma^q_{yq}= -\frac{1}{2y} &  \end{array} .
\end{equation}
To get \eqref{GM22}, we apply Remark \ref{RM10} for $\mc{X}_1$ and we get  from  the
holomorphic $\Gamma$ symbols  the associated  real Christoffel's symbols, 
see also e.g. \cite[Exercise 8 p 58]{MPC}.

{\bf{b)}}    With \eqref{geoX}, in $\nu H^2_1=kH_2$ we introduce
$H_1=\dot{u}-p\dot{v}$, $H_2=\ii \ddot{u}-\frac{\dot{u}\dot{v}}{y}$ and taking the
real and imaginary part,  we get \eqref{421a} and \eqref{421b}.

Taking the derivative of \eqref{UVPQ}, we get successively 
\[
\dot{u}=\dot{p}v+p\dot{v}+\dot{q},\quad \ddot{u}=\ddot{p}v+2\dot{p}\dot{v}+p\ddot{v}+\ddot{q}
\]
which introduced in  expression \eqref{geoX2} gives the expression
of $K_3$ in \eqref{geoYY1}.
Now in the first equation \eqref{geoYY1} we introduce  the expression of
$K_3$ and  taking  into consideration  first equation in
\eqref{geoYY2}, we get
\begin{equation}\label{dodo}
2\ii (\ddot{p}v+\ddot{q})-\bar{F}\frac{\dot{v}}{y}=0.
\end{equation}
Taking the real and imaginary part of \eqref{dodo}, we get
\eqref{421c}, \eqref{421d}.

{\bf{c )}} Assertion c) is a consequence  of a) and b).
\end{proof}

We make now
\begin{Remark}\label{ECHI}
  The expression \eqref{DD2} of the determinant of the metric matrix of
  $\mc{X}^J_1$ in the variables $(x,y,q,p)$ can be obtained
  from the expression \eqref{DETH2} of the metric matrix of
  $\mc{D}^J_1$ in the variables $(m,n,\alpha,\beta)$.
\end{Remark}
\begin{proof}
 Let $(M,g(x))\rightarrow (M',g'(x'))$ be an isometry of Riemannian
  manifolds. Then we have the relation, see e.g. \cite[(9.3) p 23]{EISEN}
 \begin{equation}\label{SCHD}
\det g'(x')=\det g(x) J^2, \text{~where~} J:= \det |\frac{\pa
      x^i}{\pa x'^j}|.
  \end{equation}
In our case of $(\mc{D}^J_1,g(m,n,\alpha,\beta))\rightarrow
(\mc{X}^J_1,g(x,y,q,p))$, with the first relation \eqref{210b} and
\eqref{PMN},  we have the change of coordinates
\begin{equation}\label{xYqP}
  (x,y,q,p)=(-\frac{2\beta}{Q},\frac{P}{Q},\frac{C}{P},\frac{D}{P}), 
\end{equation}
where $P$ ($Q$) was defined in \eqref{kk1} (respectively
\eqref{QQQ}).

With formula \eqref{SCHD} applied  to the change of coordinates \eqref{xYqP} from
the Siegel-Jacobi disk to the Siegel-Jacobi upper half-plane we have
to 
calculate the Jacobian 
\begin{equation}\label{JJ}
  I:=\frac{\pa (x,y,q,p)}{\pa (m,n,\alpha,\beta)} =I_1I_2 , \quad I_1:=\frac{\pa (x,y)}{\pa
  (\alpha,\beta)},\quad I_2:= \frac{\pa (q,p)}{\pa
  (m,n)}.
\end{equation}
 Using the Cauchy-Riemann equations
\begin{equation}\label{CR}
  \frac{\pa x}{\pa \alpha}=\frac{\pa y}{\pa \beta},\quad \frac{\pa
    x}{\pa \beta }=-\frac{\pa y}{\pa \alpha},
\end{equation}
we have to calculate 
\begin{equation}\label{II}
 I_1=\frac{\pa x}{\pa \alpha}\frac{\pa
    y}{\pa \beta}-\frac{\pa x}{\pa \beta}\frac{\pa y}{\pa \alpha}=
  (\frac{\pa x}{\pa\alpha})^2+(\frac{\pa x}{\pa \beta})^2,
\end{equation}
and with \eqref{xYqP}
$$\frac{\pa x}{\pa \alpha}=4\frac{(\alpha-1)\beta}{Q^2},\quad
\frac{\pa x}{\pa \beta}=-2\frac{(\alpha-1)^2-\beta^2}{Q^2},$$ 
introduced in \eqref{II}, we get
  $$I_1=\frac{4}{Q^2}, \quad  I_2=\frac{1}{P},\quad I=\frac{4}{Q^2}\frac{1}{P}.$$
   With \eqref{SCHD} and \eqref{JJ}, we get for $J=I^{-1}$
 $$\det g(x,y,q,p) = (\frac{k \nu}{2}\frac{Q^2}{P^2})^2,$$
  i.e.  relation \eqref{DD2} because of the notation \eqref{AK}, \eqref{xYqP}.
\end{proof}

We show below that  the geodesic equations \eqref{420} on
$\mc{X}^J_1$ are found making in \eqref{geo}  the change of variables
given in \eqref{ULTRAN} with the consequence that
\begin{Proposition}\label{PP5}
The second partial Cayley transform  \eqref{ULTRAN}
  $\Phi_1 :     (\mc{D}^J_1,\dd s^2_{\mc{D}^J_1}(w,z))
\rightarrow  (\mc{X}^J_1, \dd s^2_{\mc{X}^J_1} (x,y,p,q))$ is a
  geodesic mapping.

\end{Proposition}
\begin{proof}
From \eqref{DOTZ}, we get
\begin{equation}\label{DDOTZ} 
\ddot{z}=\frac{2\ii}{v+\ii}
\left\{Z-2\frac{\dot{v}\bar{\dot{\eta}}}{v+\ii}
-\frac{\bar{\eta}[\ddot{v}(v+\ii)-2\dot{v}^2]}{(v+\ii)^2}\right\},
\quad Z:=v\ddot{p}+\ddot{q}.
\end{equation}
With \eqref{32a}, \eqref{32b} and \eqref{DDOTZ}, we get for $G_3$ in
\eqref{geo1} the value
$G'_3:=G'_3(v,\eta):=FC_1\circ\Phi (G_3(w,z))$ the expression
\begin{equation}\label{G3P}
G'_3=\frac {2\ii}{v+\ii}\left(Z
-\frac{\bar{\eta}}{v+\ii}V_1
+2\frac{\dot{v}\bar{F}}{\bar{v}-v}\right),
\end{equation}
where  $F$ was defined in \eqref{BUVpq1} and $V_1$ in \eqref{geoX2}.

We also find for $G'_i:=G'_i(v,\eta)=\Phi_1\circ\Phi (G_i(w,z))$, $i=1,2$ the expressions
\begin{equation}\label{G1G2}
G'_1= \frac{2\ii}{v+\ii}F,\quad
G'_2=\frac{2\ii}{(v+\ii)^2}V_1.
\end{equation}
Introducing \eqref{G1G2} into \eqref{geo2}, we obtain
\begin{equation}\label{nG1G2}
V_1=-\ii
\iota^{-1}F^2.\end{equation}
The real (imaginary) part of \eqref{nG1G2} expresses  equation \eqref{421a}
(respectively \eqref{421b}) because
$\iota^{-1}=\frac{\epsilon}{2}$. 

Introducing \eqref{nG1G2} into \eqref{G3P}, we get
\begin{equation}\label{GP32}
G'_3=\frac{2\ii}{v+\ii}\left(Z+s
2\dot{v}\frac{\bar{F}}{\bar{v}-v}+
\ii\iota^{-1}\frac{\bar{\eta}F^2}{v+\ii}\right).
\end{equation}
Now we introduce \eqref{GP32} into \eqref{geo1} and we get
\[
Z+2\dot{v}\frac{\bar{F}}{\bar{v}-v}=0
\]
whose real (imaginary part) gives \eqref{421c} (respectively \eqref{421d}).
\end{proof}

  Proposition \ref{PP5} is  a more precise 
formulation of  \cite[Remark 8]{SB14}  for the $FC$-transform
\eqref{E32} on $\mc{D}^J_1$ and
  \cite[Proposition 6]{GAB} for $\mc{D}^J_n$, $n\in\N$.

\section{Geodesics on the extended Siegel--Jacobi upper half-plane}\label{GEX}

In order to get geodesic equations on the extended Siegel-Jacobi
upper half-plane in the S-variables $(x,y,p,q,\kappa)$, we use the expression  of the metric
  on   $\tilde{\mc{X}}^J_!$ given in Proposition  \ref{PROP5}.  The solution  \eqref{HEISI} 
of geodesic equations \eqref{EURJ}  on  the Heisenberg group 
$H_1$ in  Theorem \ref{PROP5} is taken form  \cite[(11) in Theorem 1]{mar}.

\begin{Theorem}\label{PROP5}

The  three-parameter   metric on  the extended  Siegel-Jacobi upper
  half-plane  
  $\tilde{\mc{X}}^J_1$ expressed  in the  S-coordinates
  $(x,y,p,q,\kappa)$,  left-invariant with  respect  to  the action  of the Jacobi group
 $G^J_1(\R)$,  is given  by  item  \emph{4)} in
  \emph{Proposition \ref{BIGTH}} as 
  \begin{equation}\label{linvG}
    \begin{split}
 {\rm d} s^2_{\tilde{\mc{X}}^J_1}(x,y,p,q,\kappa) &
 ={\rm d} s^2_{\mc{X}^J_1}(x,y,p,q)+\lambda^2_6(p,q,\kappa)\\
 &=\frac{\alpha}{y^2}\big({\rm d} x^2+{\rm d}
 y^2\big)+\left(\frac{\gamma}{y}S+\delta q^2\right){\rm d} p^2+
 \left(\frac{\gamma}{y}+\delta p^2\right){\rm d} q^2 +\delta {\rm d} \kappa^2\\
& + 2\left(\gamma\frac{x}{y}-\delta pq\right){\rm d} p{\rm d} q +2\delta (q{\rm d} p{\rm d}
 \kappa-p{\rm d} q {\rm d} \kappa).
    \end{split}
  \end{equation}

 a) The  geodesic equations on $\tilde{\mc{X}}^J_1$ associated to
 the metric \eqref{linvG} are 
 \begin{subequations}\label{EURI}
  \begin{align}
    E'_1: &= E_1 = 0, \label{EUR1}\\
    E'_2: &= E_2= 0,\label{EUR2}\\
 \label{EU3}   E'_3: &=
                       E_3+\frac{2\tau}{y}[xq\dot{p}^2+(q-px)\dot{p}\dot{q}-p\dot{q}^2+R\dot{\kappa}]
                       = 0,\quad \tau:=\frac{\delta}{\gamma},\quad \xi:=px+q,\\
\label{EU4}    E'_4: &= E_4+\frac{2\tau}{y}[-qS\dot{p}^2+(pS-xq)\dot{p}\dot{q}-S\dot{p}\dot{\kappa}+ x\dot{q}(p\dot{q}-\dot{\kappa})]= 0,\\
  \nonumber E'_5: &= \ddot{\kappa} +\frac{py^2-\xi x}{y^2}\dot{x}\dot{p}-\frac{\xi}{y^2}\dot{x}\dot{y}
            -\frac{2px+q}{y}\dot{y}\dot{p}-\frac{p}{y}\dot{y}\dot{q}\\
            & ~~~+\frac{2\te}{y}[-\dot{p}(pS+qx)(q\dot{p}+\dot{\kappa})+(p^2S-q^2)\dot{p}\dot{q}+
 \label{EU5}            \xi\dot{q}(p\dot{q}-\dot{\kappa})]= 0.
  \end{align}
  \end{subequations}
If $\tau=0 $, then  the first four equations  \eqref{EURI} are the equations
\eqref{420} of geodesics on the Siegel-Jacobi upper half-plane  with
invariant metric \eqref{METRS2}. 

b) In particular, the  geodesic equations on the Heisenberg group $H_1$
corresponding to the metric in case \emph{6)} in \emph{Proposition
  \ref{BIGTH}} are obtained as the
 particular cases \eqref{EURJ}  of  \eqref{EU3}-\eqref{EU5},
where 
\begin{subequations}\label{EURJ}
  \begin{align}
    E''_3: &= \ddot{p}+2\tau(-p\dot{q}^2+q\dot{p}\dot{q} +\dot{q}\dot{\kappa}),\\
   E''_4: &= \ddot{q}+2\tau(-q\dot{p}^2+p\dot{p}\dot{q}-\dot{p}\dot{\kappa}), \\
    E''_5: &=
                 \ddot{\kappa}+2\tau[pq(-\dot{p}^2+\dot{q}^2)+(p^2-q^2)\dot{p}\dot{q}
-(p\dot{p}+q\dot{q})\dot{\kappa}].
  \end{align}
\end{subequations}

Geodesic lines issuing from $(0,0,0)$ such that
$(\dot{x}(0),\dot{y}(0),\dot{z}(0))= $  $(r\cos\phi,r\sin\phi,\sigma)$, $\sigma\not=0$  in $H_1$ are given
by 
\begin{equation}\label{HEISI}
 \left\{
 \begin{array}{l}
x(t)=\frac{r}{2\sigma}(\sin(2\sigma t+\phi)-\sin\phi),\\
y(t)=\frac{r}{2\sigma}(\cos \phi-\cos(2\sigma t+\phi)),\\
z(t)=\!\frac{1+\sigma^2}{2\sigma}t-\frac{1-\sigma^2}{4\sigma^2}\sin
2\sigma t,
\end{array}
 \right.
\end{equation}
while if $\sigma =0$, 
$$(x(t),y(t),z(t))=(\alpha_1 t,\beta_1 t,0),\quad \alpha_1^2+\beta_1^2=1.$$

\end{Theorem}
\begin{proof}
  
 a) The metric matrix associated  to the  metric \eqref{linvG} is
\begin{equation}\label{begGG}
g_{\tilde{\mc{X}}^J_1}\! = \!\left(\begin{array}{ccccc}g_{xx} &0 &0 &0&0\\
0& g_{yy}& 0& 0 & 0\\
0& 0& g_{pp} & g_{pq} & g_{p\kappa}\\0 & 0& g_{qp}& g_{qq}
 &g_{q\kappa}\\
0& 0& g_{\kappa p}& g_{\kappa q} & g_{\kappa\kappa}
 \end{array}\right),\!
 \begin{array}{cc}g_{xx}\!=\frac{\alpha}{y^2}, &
  \!g_{yy}\!=\!\frac{\alpha}{y^2},\\
 g_{pq}\!=\!\gamma\frac{x}{y}-\delta p q , &
~g_{p\kappa}\!=\!\delta q, g_{q\kappa}\!=\!-\delta p, \\
   g_{pp} \!=\!\gamma\frac{S}{y}+\delta q^2,&
 g_{qq}\!=\!\frac{\gamma}{y}+\delta p^2,   g_{\kappa\kappa}\!=\!\delta.
\end{array}
\end{equation}
and
\[
\det(g_{{\tilde{\mc{X}}}^J_1})=\delta\left(\frac{\alpha\gamma}{y^2}\right)^2.
\]

 The invers of the metric matrix \eqref{begGG} is
 \[
   g^{-1}_{\tilde{\mc{X}}^J_1}\!\!= \!\!\left(\begin{array}{ccccc}g^{xx} &0 &0 &0&0\\
0& g^{yy}& 0& 0 & 0\\
0& 0& g^{pp} & g^{pq} & g^{p\kappa}\\0 & 0& g^{qp}& g^{qq}
 &g^{q\kappa}\\
0& 0& g^{\kappa p}& g^{\kappa q} & g^{\kappa\kappa}
 \end{array}\right),\!
\begin{array}{lll}g^{xx}\!\!=\!\!g^{yy}\!=\!\frac{y^2}{\alpha} \!\!
  &\quad 
                                                                      g^{pq}\!\!=\!\!-\!\frac{x}{\gamma
                                                                      y} \quad
   & g^{q\kappa}\!=\!\frac{pS+qx}{\gamma y} \\
  g^{p\kappa}\!=\!-\frac{\xi}{\gamma y} &
g^{pp} \!=\!\frac{1}{\gamma y}~~ 
   g^{qq}\!=\!\frac{S}{\gamma y} 
 &g^{\kappa\kappa}\!=\!\frac{1}{\delta}\!+\!\frac{\xi^2\!+\!(py)^2}{\gamma
   y}. 
 \end{array}
 \]
With
 formula \eqref{geoI} we determine the Christofell's symbols corresponding to
 the Riemannian metric \eqref{linvG} of the extended Siegel-Jacobi
 upper half-plane. In formulas below we have  included only  the
 $\Gamma$-s which   are not  given in \eqref{GSC}
\begin{equation}\label{MNM}
\begin{array}{llllll}
 \!\Gamma^{p}_{pp} \!=\! 2\tau\frac{xq}{y}\!&\!
 \Gamma^{p}_{pq} \!=\!\tau\frac{q\!-\!px}{y}
  \!&\!\Gamma^{p}_{p\kappa} \!=\!\tau\frac{x}{y} \!&\!\Gamma^{p}_{qq}
                                               \!\!=\!\!-\!2\tau\frac{p}{y} \!&\!
  \Gamma^{p}_{q\kappa}\!=\!\tau\frac{1}{y}\! &\\
 \!\Gamma^{q}_{pp} \!=\!-2\tau\frac{qS}{y}
  \!&\!\Gamma^{q}_{pq}\!\!=\!\!\tau\frac{-xq\!\!+\!\!pS}{y}
  \!&\!\Gamma^{q}_{p\kappa}\!=\!-\!\tau\frac{S}{y}
 \!&\!\Gamma^{q}_{qq}\!=\! 2\tau
 \!\frac{xp}{y} \!&\!\Gamma^{q}_{q\kappa}\!=\!-\!\tau\frac{x}{y}\!& \\
\!\Gamma^{\kappa}_{xp}\!=\!\frac{py^2\!-\!x\xi}{2y^2}& \!\Gamma^{\kappa}_{xq}\!=\!-\!\frac{\xi}{2y^2}\!&  \Gamma^{\kappa}_{yp}\!\!=\!\!-\!\frac{2px+q}{2y}
  \!&\!\Gamma^\kappa_{yq}\!=\!-\frac{p}{2y} &
\Gamma^{\kappa}_{pp} \!=\!
                                              -2\tau\frac{q}{y} (pS+qx)\!& \\
  \!\Gamma^{\kappa}_{pq} \!=\! \tau\frac{p^2S-q^2}{y}& 
\!\Gamma^{\kappa}_{p\kappa} \!\!=\!-\!\tau\frac{pS+qx}{y}\!
&
 \!\Gamma^{\kappa}_{qq} \!\!=\!\!2\tau\frac{p\xi}{
                                                              y}\!\!&\!\Gamma^{\kappa}_{q\kappa}
                                                                      \!\!=-\!\tau\frac{\xi
                                                           y}{y}.&
                                             & 
 \end{array}
\end{equation}
\eqref{MNM} implies the geodesic equations on $\tilde{\mc{X}}^J_1$ given in
\eqref{EURJ}. 

b)  According to \eqref{MPQ}, the Heisenberg group $\rm{H}_1$
embedded into \Spp~ corresponds to $(x,y,\tau)=(0,1,0)$  in \eqref{EURI}. 

Indeed, the matrix associated to the fundamental two-form
\eqref{MTRSINV} is 
\begin{equation}\label{NUMA}
g_{H_1}\! = \!\left(\begin{array}{ccc}g_{\lambda\lambda} &g_{\lambda\mu}&g_{\lambda\kappa}\\
g_{\mu\lambda}& g_{\mu\mu}& g_{\mu\kappa}\\
g_{\kappa \lambda}&  g_{\kappa \mu} & g_{\kappa\kappa} 
 \end{array}\right),\!
 \begin{array}{ccc}g_{\lambda\lambda}\!=a_1+a_3\mu^2, &
 \!g_{\mu\mu}\!=a_2+a_3\lambda^2& g_{\kappa\kappa}=a_3,\\
 g_{\lambda\mu}=-a_3\lambda\mu,&
 g_{\lambda\kappa}=a_3\mu,& g_{\mu\kappa}=-a_3\lambda,
\end{array}
\end{equation}
and 
\[
  \det(g_{H_1})=a_1a_2a_3.
\]
The inverse  of the metric matrix $g_{H_1}$  \eqref{NUMA} is
\begin{equation}\label{GGegGG}
g^{-1}_{H_1}\! = \!\left(\begin{array}{ccc}g^{\lambda\lambda} &g^{\lambda\mu}&g_{\lambda\kappa}\\
g^{\mu\lambda}& g^{\mu\mu}& g^{\mu\kappa}\\
g^{\kappa \lambda}&  g^{\kappa \mu} & g^{\kappa\kappa} 
 \end{array}\right),\!
 \begin{array}{ccc}g^{\lambda\lambda}\!=\frac{1}{a_1}, &
 \!g^{\mu\mu}\!=\frac{1}{a_2}& g^{\kappa\kappa}=\frac{1}{a_3}+(\frac{\lambda^2}{a_2}-\frac{\mu^2}{a_1}),\\
 g^{\lambda\mu}=0,&
   g^{\lambda\kappa}=-\frac{\mu}{a_1}& g^{\mu\kappa}=\frac{\lambda}{a_2}.
\end{array}
\end{equation}
The Christoffel's symbols associated to the metric matrix \eqref{NUMA}
are 
\[
\begin{array}{llllll}
\Gamma^{\lambda}_{\lambda\lambda} = 0& 
\Gamma^{\lambda}_{\lambda\mu}=\frac{a_3}{a_1}\mu&
\Gamma^{\lambda}_{\lambda\kappa} =\!0\!& 
\Gamma^{\lambda}_{\mu\mu} \!=\!-2\frac{a_3}{a_1}\lambda&
 \Gamma^{\lambda}_{\mu\kappa}\!=\!\frac{a_3}{a_1}&   
\Gamma^{\lambda}_{\kappa\kappa} \!=\!0 \\ 
\Gamma^{\mu}_{\lambda\lambda} \!=\!-2\frac{a_3}{a_2}\mu& 
\Gamma^{\mu}_{\lambda\mu} \!=\frac{a_3}{a_2}\lambda& 
\Gamma^{\mu}_{\lambda\kappa} \!=\!-\frac{a_3}{a_2}&
\Gamma^{\mu}_{\mu\mu} =\!0\!& 
\Gamma^{\mu}_{\mu\kappa} \!=\!0& 
\Gamma^{\mu}_{\kappa\kappa} \!=\!0\!\\
\Gamma^{\kappa}_{\lambda\lambda} \!=\!-\!2\frac{a_3}{a_2}\lambda\mu&
\!\Gamma^{\kappa}_{\lambda\mu} \!=\!a_3(\frac{\lambda^2}{a_2}\!-\!\frac{\mu^2}{a_1})&
\!\Gamma^{\kappa}_{\lambda\kappa} \!\!=\!-\!\!\frac{a_3}{a_2}\lambda\!& 
\Gamma^{\kappa}_{\mu\mu} \!=\!2\frac{a_3}{a_1}\lambda\mu\!& 
\Gamma^{\kappa}_{\mu\kappa} \!=\!-\frac{a_3}{a_1}\mu\!& 
\Gamma^{\kappa}_{\kappa\kappa} \!=\!0
\end{array}
\]
The geodesic equations on $\rm{H}_1$ corresponding to the
three-parameter invariant metric
\eqref{MTRSINV} are
\begin{subequations}\label{CCXX}
  \begin{align}
    ~~&
        \ddot{p}+2\frac{a_3}{a_1}(-p\dot{q}^2+q\dot{p}\dot{q}+\dot{q}\dot{\kappa})=0,\\
    ~~&\ddot{q}+2\frac{a_3}{a_1}(-q\dot{p}^2+p\dot{p}\dot{q}-\dot{p}\dot{\kappa})=0,\\
    ~~& \ddot{\kappa}+2a_3[\frac{pq}{a_2}(\dot{q}^2-\dot{p}^2)+(\frac{p^2}{a_2}-\frac{q^2}{a_1})\dot{p}\dot{q}-\frac{1}{a_2}(p\dot{p}+q\dot{q})]=0.
  \end{align}
\end{subequations}
We have the following covariant derivatives associated  to the
3-parameter invariant metric \eqref{MTRSINV} of $\rm{H}_1$
\begin{subequations}\label{EQQ3}
  \begin{align}\nabla_{p}\pa_{p}&=-2\frac{a_3}{a_2}qL^q,
                                  \quad\nabla_p\pa_{\kappa}=\nabla_{\kappa}\pa_p=-\frac{a_3}{a_2}L^q,\quad
                                  \nabla_{\kappa}\pa_q=\nabla_q\pa_{\kappa}=\frac{a_3}{a_1}L^p,\\
    \nabla_{p}\pa_{q}& =\nabla_{q}\pa_p=a_3[q(
                             \frac{1}{a_1}\pa_p+\frac{1}{a_2}\pa_q)+(\frac{p^2}{a_2}-\frac{q^2}{a_1})\pa_{\kappa}].
  \end{align}
\end{subequations}
With \eqref{LEFT11} and   \eqref{EQQ3}, we get the covariant
derivatives of the Riemannian connection of the left-invariant metric
\eqref{MTRSINV} with the ortonormal frame \eqref{LEFT11}
\begin{subequations}\label{EQQ44}
  \begin{align}
    \nabla_{L^p}L^p& =0,\quad\nabla_{L^p}L^q=L^{r},\\
    \nabla_{L^p}L^{r}& =-\frac{a_3}{a_2}L^q\quad \nabla_{L^q}L^{r}=\frac{a_3}{a_1}L^p.
    \end{align}
  \end{subequations}
  Equations \eqref{EQQ3} for $ a_1=a_2=a_3=1$ is Lemma 1 in \cite{mar}
  and we can apply Theorem 1 in  \cite{mar}, which is the content of last assertion
  in Theorem \ref{PROP5}. 

We take $a_1=a_2=\gamma$, $a_3=\delta$ in \eqref{CCXX} or in \eqref{EQQ44} we get
the geodesic equations on $\rm{H}_1$ given  in \eqref{EURJ}.
\end{proof}

\section{Appendix: Geodesic mappings}\label{gm}
In  several places of this paper, namely in:
 Proposition \ref{P3}, item c)  for the partial Cayley transform $\Phi$
\eqref{210}, Remark \ref{RE3}, item c) for the Cayley transform, Proposition
\ref{P4} for the change of coordinates \eqref{UVPQ} and in 
Proposition \ref{PP5} for the second partial Cayley transform
\eqref{ULTRAN} $\Phi_1$, we proved that the mentioned
applications are geodesic mappings. 

Bellow   we collect several well known facts about geodesic mappings.

Firstly we recall the notion of {\it isometry}  \cite[p 60]{helg} (or
{\it motion}
\cite[Section 1.2]{alx})  between Riemannian
manifolds.

Let us suppose that    two Riemannian spaces
  $(M_n,g)$,
  $(\tilde{M}_n,\tilde{g})$ have the fundamental forms
  \cite[(40.2)]{EISEN}
\begin{equation}\label{2metr}
\dd s^2_{M_n}=g_{ij}\dd x_i\dd x_j, \quad  \dd
s^2_{\tilde{M}_n}=\tilde{g}_{ij}\dd x_i\dd x_j,
\end{equation} and $f$ is a function
\begin{equation}\label{DEFF}
  f: ~(M_n,g)\rightarrow (\tilde{M}_n,\tilde{g}),\quad
  y=f(x),\end{equation}
such that the Jacobian
  \begin{equation}\label{JAK}J=\det(\frac{ \pa y_i}{\pa
      x_j})_{i,j=1,\dots, n}\not= 0.\end{equation}
  With   \cite[pages 22-24]{helg},  \cite[p 161]{kn1},  \cite[(7.10) p 18]{EISEN} and \cite[(2.15) p 59]{mikes}
\begin{Proposition}\label{oare}
  Let \eqref{DEFF}
  be an isometry, or motion,  between the    Riemannian spaces
  \eqref{DEFF}, i.e.  $f$ is an diffeomorphism of $M_n$
onto $\tilde{M}_n$ and 
\begin{equation}\label{FstarG}
  f^*\tilde{g}=g,
\end{equation}
or 
 \begin{equation}\label{GPXY}g_p(X,Y)=\tilde{g}_{f(p)}(\dd
f_p(X),\dd f_p(Y)), \quad \forall ~p\in M_n, ~\forall X,Y\in
(TM_n)_p.\end{equation}
If $$X_x:= X^i(x)\frac{\pa }{\pa x_i}, \quad Y_x:= Y^i(x)\frac{\pa }{\pa
  x_i},\quad X^*_{f(p)}:=\dd f_p (X_p),$$
 then $$X^i_x=(X_{f(p)}^*)^{\alpha}\frac{\pa x^i}{\pa y^{\alpha}},$$
 and \eqref{GPXY} becomes
 \begin{equation}\label{GIJX}
 g_{ij}(x)=\frac{\pa y^{\alpha}}{\pa
   x^{i}}\tilde{g}_{\alpha\beta}(f(x))\frac{\pa y^{\beta}}{\pa x^j},\quad
 \tilde{g}_{\alpha\beta}(f(x))= \frac{\pa
   x^{i}}{\pa y^{\alpha}} g_{ij}(x) \frac{\pa x^j}{\pa y^{\beta}}, 
\end{equation}
i.e. $g_{ij}$ is a $(0,2)$-contravariant tensor.
\end{Proposition}

Now we recall  \cite[ Definition 5.1 p 127]{mikes}
\begin{deff}
If $f$ \eqref{DEFF} is a diffeomorphism between manifolds with
affine connections, then $f$ is called a {\it geodesic mapping} if it maps
geodesic curves on $M_n$ into geodesic curves on $\tilde{M}_n$.
\end{deff}

In fact, there are three methodes for proving that a mapping
\eqref{DEFF} is a  geodesic one.

a) By {\it brute force-calculation}, i.e.:\\
1) suppose that we have  geodesic equations 
$G(x)$ on the manifold $M_n$ in the variables $x$;\\
2) we make the change of coordinates \eqref{DEFF} in $G(x)$ and we get  
a system of  differential equations  $H(y)$ on $\tilde{M}_n$; \\
3) we verify that $H(y)$ are
 geodesic equations on $\tilde{M}_n$ corresponding to the 
Christoffel symbols associate to $\nabla_{\tilde{g}}$.

In the present paper
we have applied this method.

b)  {\it Any motion
group takes geodesic into a geodesic} \cite[Section 1.2 p 9]{alx}.

We
make some more comments.

 - 
  If $f$ is an isometry of a Riemannian manifold,
  then $f$ preserves distances \cite[p 60]{helg}. It is proved in \cite[Theorem
  11.1]{helg} that: Let $M$  be a Riemannian manifold and $f$ a
  distance-preserving mapping onto itself. Then $f$ is an isometry.

  -  One can argue \cite[p 60]{helg}, \cite{ex1}  that the isometry $f$ also preserves the induced
  distances  $d_1$, $d_2$ on $M$, $\tilde{M}$ from $g$, $\tilde{g}$ respectively,  that is $
  d_1(x,y)=d_2(f(x),f(y))$ for $x,y \in M$. It is  easy to show \cite{ex1} 
  that $f$ sends geodesics on $M$ to geodesics on $\tilde{M}$, using the length
  minimizing property of geodesics and that $f$ is
  distance-preserving.
  
  c)  {\it the Levi-Civita} \cite{LC} equations  \eqref{LCE}.

  In \cite[Section 40 ``Spaces with corresponding geodesics'',  pages
  131-133]{EISEN}, Eisenhart considers the two Riemannian spaces $(M_n,g)$,
  $(\tilde{M}_n,\tilde{g})$ with the fundamental forms \eqref{2metr}
  and  Christofell's symbols
 $\Gamma^l_{ij}$, (respectively $\tilde{\Gamma}^l_{ij}$).

Based on \cite[(40.6), (40.8), (11.4)]{EISEN}
\begin{Proposition}\label{PROP55}
  The $\frac{n^2(n+1)}{2}$ equivalent conditions for the spaces  $M_n, \tilde{M}_n$ to have 
  the same  
  geodesics are the Levi-Civita equations
  \begin{subequations}\label{LCE}\begin{align}
        \tilde{\Gamma}^k_{ij} & =   \Gamma^k_{ij}
                                +\psi_i\delta^k_j+\psi_j\delta^k_i,\quad
                                i,j,k=1,\dots,n; \label{LCE1}
                                                \\                      
                                 \tilde{g}_{ij,k}  & = 2
                                 \psi_k\tilde{g}_{ij} +
                                 \psi_i\tilde{g}_{jk} +
                                 \psi_j\tilde{g}_{ik}\label{LCE2},
                               \end{align}
                             \end{subequations}
                             where $\tilde{g}_{ij,k}$ is the covariant
                             derivative  of
                             $\tilde{g}_{ij}$ with respect to $x^k$
                             and the fundamental tensor $g_{ij}$ of
                             type $(0,2)$, i.e. 
                             \begin{equation}\label{CVD}
                               \tilde{g}_{ij,k}=\pa_k\tilde{g}_{ij}-\Gamma^l_{ki}\tilde{g}_{lj}-\Gamma^l_{kj}\tilde{g}_{il},
                             \end{equation}
                             and
                    \begin{equation}\label{derP}
                               \psi_i=\pa_i\Psi, \quad   \Psi=\frac{1}{2(n+1)}\ln
|\frac{\det{\tilde{g}}}{{\det g}}|.\end{equation}
\end{Proposition}

- For equations \eqref{LCE1}  see also
\cite[(12) p 157, (72) p 270]{vr} and for \eqref{derP} see \cite[p
270]{vr}. Note that the Levi-Civita condition \eqref{LCE2} appears in
\cite[Theorem 94.1 p 290]{krey} in the case of geodesic mappings of
constant Gaussian curvature between  a portion of a surface $S$ onto
a portion of another surface $S^*$.

- With Proposition \ref{oare}, we get also   the well-known relation \eqref{SCHD}
   which allows us  write down \eqref{derP} as 
  \begin{equation}\label{derP2}
    \Psi=-\frac{1}{n+1}\ln |J|.
\end{equation}
 
- In  \cite[Section 7.1, p
167]{mikes} it is presented the particular
case considered in Proposition \ref{PROP55} of two Riemannian manifolds $V_n=(M,g)$ and
$\tilde{V}_n=(\tilde{M},\tilde{g})$ with corresponding Riemannian
connections $\nabla$ and $\tilde{\nabla}$, respectively. It is supposed
that there exists a geodesic map $f: ~M\rightarrow \tilde{M}$  and
$ \tilde{M}$ is identified via $f$ with $M$ as in \cite[Section 3.1]{mikes}.
Then $V_n$ admits a geodesic mapping onto  $\tilde{V}_n$ if and only
if the  equivalent Levi-Civita equations \eqref{LCE}  hold.

\subsection*{Acknowledgements}

The author is grateful  to Dr. E.M. Babalic for verifying with
{\it Mathematica} the  Christofell symbols   and for improving the
presentation. This research  was conducted in  the  framework of the 
ANCS project  program   PN 19 06
01 01/2019. 


\end{document}